\documentclass[12pt]{article}
\usepackage{ulem}
   \usepackage[dvips]{graphicx}
   \usepackage{makeidx, comment}
   \usepackage{psfrag}
\usepackage{pdfsync}
\usepackage{hyperref}
   \usepackage{psfrag}
   \usepackage{amssymb,amsmath,amsthm, graphicx}
   \usepackage[LY1]{fontenc}				
\usepackage{times}
\usepackage{amscd, amsfonts, tikz}
\usepackage{color}
\newtheorem{theorem}{Theorem}[section]
\newtheorem{lemma}[theorem]{Lemma}

\newtheorem{proposition}[theorem]{Proposition}
\newtheorem{corollary}[theorem]{Corollary}

\def\beq{\begin{equation}}
\def\eeq{\end{equation}}

\def\le{\leq}

\def\cov{{\rm Cov}}

\def\var{{\rm Var}}

\def\bI{\mathbf{1}}

\def\cA{\mathcal{A}}
\def\cB{\mathcal{B}}

\def\cH{\mathcal{H}}
\def\cI{\mathcal{I}}
\def\cJ{\mathcal{J}}

\def\cM{\mathcal{M}}

\def\cO{\mathcal{O}}

\def\cR{\mathcal{R}}
\def\cS{\mathcal{S}}

\def\tf{\tilde{f}}

\def\hW{\tilde{W}}

\def\tmu{\tilde\mu}

\newcommand{\eps}{\varepsilon}

\def\sab{\mathcal Z_{\alpha,\xi,s}}
\def \E{\mathbf{E}}
\def \P{\mathbf{P}}

\def\({\left(}
\def\rt){\right)}
\def\[{\left[}
\def\]{\right]}

\def \D{\mathbb{D}}

\def \R{\mathbb{R}}

\def \N{\mathbb{N}}

\def \P{\mathbf{P}}
\def \E{\mathbf{E}}
\def \aa{\alpha}

\def \eps{\epsilon}

\def \sign{\text{sign}}
\def \ff{\infty}

\def \({\left(}
\def \){\right)}

\def \lb{\left|}
\def \rb{\right|}

\def \beq{\begin{equation}}
\def \ee{\end{equation}}
\def \bea{\begin{eqnarray}}
\def \eea{\end{eqnarray}}
\def \bes{\begin{eqnarray*}}
\def \ees{\end{eqnarray*}}

\def \nn{\nonumber}

\numberwithin{equation}{section}

\begin{document}

\title{Convergence to $\alpha$-stable L\'evy motion for chaotic billiards with several cusps at flat points
}

\author{Paul Jung$^{\,\rm 1}$ and Fran\c{c}oise P\`ene$^{\,\rm 2}$ and Hong-Kun Zhang$^{\,\rm 3}$\\
{\footnotesize{$^{\rm 1}$Department of Mathematical Sciences, KAIST, Daejeon, South Korea}}\\
{\footnotesize{$^{\rm 2}$Universit\'e de Brest, Institut Universitaire de France,
		LMBA, UMR CNRS 6205, Brest, France}}\\
{\footnotesize{$^{\rm 3}$Department of Mathematics, University of Massachusetts, Amherst, MA, USA}}}
%

\maketitle

\begin{abstract} We consider billiards with several {possibly non-isometric and asymmetric} cusps at flat points; the case of a single {symmetric} cusp was studied previously in \cite{Z2016b} and \cite{JZ17}.  In particular, we show that properly normalized Birkhoff sums of H\"older observables, with respect to the billiard map, converge in Skorokhod's $M_1$-topology to an $\alpha$-stable L\'evy motion, where  $\alpha$  depends on the `curvature' of the flattest points and the skewness parameter $\xi$ depends on the values of the observable at those same points.  
Previously, \cite{JZ17} proved convergence of the one-point marginals to totally skewed $\alpha$-stable distributions {for a symmetric cusp}. The limits we prove here are stronger, since they are in the functional sense, but also allow for more varied behaviour due to the presence of multiple cusps. In particular, the general limits we obtain allow for any skewness parameter, as opposed to just the totally skewed cases. We also show that convergence in the stronger $J_1$-topology is not possible.\end{abstract}

\tableofcontents

\section{Introduction}

 The origins of the modern theory of hyperbolic dynamical systems lie in classical and statistical
mechanics through the study of ergodic and statistical properties. Indeed understanding
statistical properties and proving various probability limit theorems are vital in the study of statistical mechanics.   For example, such studies shed light on the important issue of fluctuations in
entropy production. This paper is a contribution in this direction by showing a mechanism by which an hyberbolic dynamical system can converge on the process level to something other than a Brownian motion.

Begin by letting $T$ be a measure-preserving transformation on $(\cM,\mu)$ and
$\{{X}_n=f\circ T^n\}$ the stochastic process generated by the system for an
observable $f$. One typical goal is to prove the (functional) central limit theorem for the normalized partial-sum process associated to $\{{X}_n\}$ which is itself generated by a  mixing hyperbolic system, i.e., to prove that, weakly 
$$\left(\frac{1}{\sqrt{n}}\sum_{j=1}^{[nt]} {X}_j\right)_{t\ge 0}\to\left( \sigma B(t)\right)_{t\ge 0}$$ where $(B(t))_{t\ge 0}$ is a standard Brownian motion. In other words, the (normalized) partial-sum process converges in distribution, in $C[0,1]$, to a Brownian motion with diffusion parameter $\sigma>0$. The finiteness of the variance/diffusion parameter requires that the covariances $\cov({X}_1,{X}_n)$ are summable in $n\in \mathbb{N}$. However, when the covariance sequence is not summable, the usual central limit theorem for the partial-sum process fails; and there are very few results, for hyperbolic systems, which investigate what happens under such failure.

Billiards with cusps were first introduced in the early 80's in \cite{machta1983power}, but it was not until \cite{chernov2007dispersing} that a rigorous polynomial decay of correlations of $\cO(1/n)$ was proved for Machta's original model (where cusps are formed by tangential circles).  This polynomial decay of correlations led to a nonstandard central limit theorem where a normalization of $\sqrt{n\log n}$ was used instead of $\sqrt{n}$ (see also \cite{balint2011limit}).

In \cite{Z2016b}, the author was able to construct a hyperbolic billiard model with arbitrarily slow decay rates of correlations, of order $$\cov({X_1,X_n})=\cO(n^{-a}),$$ with $a\in (0,1)$
depending on the curvature at the cusp's vertex.  In \cite{JZ17}, the first and last authors were able to prove that due to the slow decay of correlations, instead of a central limit theorem, rather, a stable limit theorem holds for the partial-sum process generated by a billiard with a single
symmetric cusp at a flat point. In particular, they showed that the limiting distribution is a totally skewed $\alpha$-stable law. In this paper, we first investigate the convergence to a stable law, for the case when the billiard table has several, possibly asymmetric, cusps at flat points. We then extend these results to a functional limit theorem in Skorokhod space, i.e.,  convergence to an $\alpha$-stable L\'evy motion. In the dynamical systems literature, functional convergence to L\'evy motion has been shown for expanding maps (see for instance \cite{tyran2010weak}), however this is the first result of this kind for 2-d hyperbolic billiards that we are aware of (after finishing this work, we were made aware of the concurrent work \cite{melbourne2018convergence} where a  functional limit theorem is also proved for the model in \cite{JZ17}.).

Let us now describe in detail the billiards with cusps at flat points that we consider here.  The dispersing billiard table $Q$ is a bounded domain of $\mathbb R^2$, the boundary of which consists of  $q\geq 3$ dispersing $C^3$ smooth curves 
$\partial Q= \cup_{i\in \mathbb Z/q\mathbb Z} \Gamma_i$, numbered in clockwise order. 
The intersections $P_i=\Gamma_i\cap\Gamma_{i+1}$ of two consecutive curves $\Gamma_i$ and
$\Gamma_{i+1}$ consist of either standard corner points (i.e. the tangent lines at $P_i$ to $\Gamma_i$ and $\Gamma_{i+1}$ do not coincide)  or `cusps' at flat points. This means that, in an appropriate euclidean coordinate system $(s,z)$ originated at some point $P_i$, the two curves $\Gamma_i$ and $\Gamma_{i+1}$ can be represented respectively
as $z=z_{i,+}(s)$ and $z=z_{i,-}(s)$ with $z_{i,\pm}$ differentiable and satisfying
\beq\label{z1s}
z_{i,\pm} (s)=\pm \frac{C_{i,\pm}}{\beta_{i}}s^{\beta_{i}}+\cO (s^{2\beta_i-1})\quad\mbox{and}\quad
z'_{i,\pm} (s)=\pm C_{i,\pm}s^{\beta_{i}-1}+\cO (s^{2\beta_i-2})\, , \quad\mbox{for}\ s\in [0,\eps]
\eeq
with $\eps>0$ some small fixed number, $C_{i,-},C_{i,+}\ge 0$ not both null and with $\beta_i\ge 2$. We will say that such a cusp at $P_i$ has flatness $\beta_i$.    
We also define
$$\bar C_i:=(C_{i,+}+C_{i,-})/2\, .
$$
We assume moreover that, for every cusp at some $P_i$, the (unique) tangent trajectory coming out of the cusp at $P_i$
hits $\partial Q$ outside of any another cusp.
Also, all boundary components are assumed to be dispersing and have curvature bounded away from zero except at the cusps $P_i$'s.
Throughout the paper, we assume that
\begin{align}\label{eq:beta star}
\beta_*:=\max(\beta_i: P_i \text{ is a cusp})>2\, .
\end{align}
The functional limit theorem we prove here will depend only on the observable near the `maximally flat' points $P_i$ 
(the heuristic reason for this is that the limit theorem is dominated by long sequences of collisions inside the various cusps, and the flattest cusps trap the longest sequences ). 
Let $\mathcal J$ be the set of labels $i$ such that $P_i$ is a cusp and let 
$\mathcal J_*$ be the set of labels $i\in\mathcal J$ such that 
$\beta_i=\beta_*$.


\begin{figure}[h]
\centering
 \includegraphics[height=2.2in]{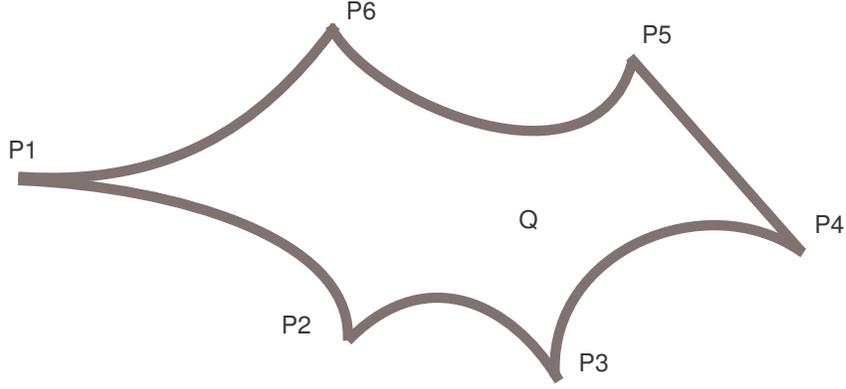}
 \caption{A billiard table with several cusps at flat points.}\label{Fig1}
 \end{figure}


The usual billiard flow is defined on the unit sphere bundle $Q\times\mathbf{S}^1$ and preserves
the Liouville measure, see \cite{chernov2006chaotic} for details. We consider the natural cross section $\cM\subset Q\times	 \mathbf{S}^1$ made of all post-collision vectors based at the boundary of the table $\partial Q$.
Any post-collision vector $x\in \cM$ can be represented by $x=(r, \varphi)$, where $r\in\mathbb R/|\partial Q|\mathbb Z$ is 
the {clockwise} curvilinear abscissa
along $\partial Q$ (where $|\partial Q|$ denotes the length of $\partial Q$),	 and $\varphi\in [0, \pi]$ is the angle formed by  the tangent line of the boundary and the collision vector in the clockwise direction. Therefore the collision space $\cM$ is identified with 
$\mathbb R/|\partial Q|\mathbb Z \times [0,\pi]$.
The corresponding  billiard map $  T : \cM\to \cM$ takes a vector $x\in \cM$ to the next post-collision vector	 along the trajectory of $x$.
Let the set $S_0$ consist of all grazing collision vectors with walls as well as all collision vectors at corner points/cusps. Then $\cS_0:=S_0\cup   T       ^{-1}S_0$ is  the singular set of $T$.
The billiard map $T: \cM\setminus	\cS_0\to \cM\setminus   T       \,\cS_0$ is a local $C^{2}$ diffeomorphism and preserves a natural
absolutely continuous probability measure $$d\mu=\frac{1}{2|\partial Q|} \sin\varphi\, dr\, d\varphi$$ on the collision space $\cM=\{(r,\varphi)\in\mathbb R/|\partial Q|\mathbb Z\times[0,\pi]\}$.

A recently popular approach for studying the statistical properties of $(T  ,\cM)$ is to use an inducing scheme as introduced in \cite{markarian2004billiards, chernov2005billiards}. By removing spots with weak hyperbolicity from the phase space, one considers the first return map on some subspace $M\subset \cM$.   For any $x\in M$,  the {first return time function} is defined by
$$\cR(x) := \min\{n\geq	 1:   T^n(x)\in M\}$$ and the  {\em (first) return map} $F\colon M \to M$ is defined by
\beq \label{Fdef}
F(x) :=   T  ^{\cR(x)}(x),\,\,\,\,\,\text{for all } x\in M.
\eeq
The return map $F$ preserves the conditional measure $\tmu:=\frac{1}{\mu(M)}\mu|_M$ and we define an  {\em induced} function by
\beq\label{inducef}
\tf(x) :=\sum_{k=0}^{\cR(x)-1}
f(T^k x), \quad x\in M.
\eeq
Here we define the subset $M$ to be the subset of $\cM$  which consists of all collisions occurring on
\begin{equation}\label{Mdef}
\partial Q\setminus(\cup_{i\in\mathcal J} B_\epsilon(P_i))
\end{equation}
where $B_\epsilon(P_i)\cap\partial Q$ admits a representation as in \eqref{z1s} and where $\epsilon$ is small enough so that
a billiard trajectory cannot go directly from one $B_\epsilon(P_i)$
to another $B_\epsilon(P_i)$ without hitting $\partial Q\setminus(\cup_{k\in\mathcal J} B_\epsilon(P_k))$.
We decompose {$$M=M_0\cup \bigcup_{i\in\mathcal J} M_i$$ so that for $i\in\mathcal J$, $M_i$ is made up of those points in $M$ 
	whose first collision inside $\cup_{j\in\mathcal J}B_\epsilon(P_j)$} (see \eqref{Mdef}),  under the map $F$, occurs
 inside $B_\epsilon(P_i)$. Moreover, $M_0=M\setminus \cup_{i\in\mathcal J} M_i$ consists of the remaining points which do not immediately enter any cusps.

Rigorous bounds on the decay of correlations for billiards with  flat points were derived
in \cite{Z2016b}, where a more detailed description of
billiards with flat points is also given. It was shown that if $f, g$ are H\"older continuous functions
 on the collision space $\cM$, then for all $n\in \mathbb{Z}$,
\beq\label{cCn}\mu(f\circ   T^n\cdot g)-\mu(f)\mu(g)= \cO(n^{-\frac{1}{\beta_*-1}}).\eeq
Here we use the standard notation $\mu(f)=\int_{\cM} f\,d\mu$, which we henceforth assume to be $0$ for $f$ and $g$. As already mentioned, it is the above slow decay of correlations that led to a stable limit theorem in \cite{JZ17}, rather than the standard central limit theorem,
for a normalized version of the Birkhoff sums
$$ \cS_n f := f + f\circ   T   +\cdots +f\circ   T  ^{n-1}\, ,$$
where $f$ is a H\"older function on $\cM$ and $f\neq 0$ at the cusp's vertex. Our results will show that for a general dispersing billiard table with multiple cusps, only the $|\mathcal J_*|\geq 1$ maximally flat cusps will contribute to the limit of properly normalized Birkhoff sums. In other words, we show that cusps with smaller order flatness will not `contribute' to the limiting stable law (nor to the functional-level convergence to a L\'evy process). Although the analysis in \cite{JZ17} can be carried to a general table to study statistical properties for observables supported  on any individual symmetric cusp, the convergence to a stable law jointly for multiple asymmetric cusps here is rather different and requires significant adjustments to the original arguments. 
%

In order to consider a functional limit theorem, denote the process $\{W_n(t):t\geq 0\}$, $n\geq 1$ by
\beq\label{defnXnt}
W_n(t):=\sum_{j=0}^{[nt]-1}\frac{ f\circ T^j}{n^{1/\alpha}},\,\,\,\,\text{for } t\geq 0,\quad\mbox{where}\ \alpha:=\frac{\beta_*}{\beta_*-1}\, .
\eeq
One can check that $\alpha \in (1,2)$ since $\beta_*>2$. 
In \cite{melbourne2015weak}, a general result was proved for obtaining functional convergence of dynamical systems to an $\alpha$-stable L\'evy motion by ``lifting''
such a limit law from an induced dynamical system to the original
system. 	Our functional limit theorem for $\{W_n(t):t\ge 0\}$, $n\geq 1,$ will be proved in two steps: first showing that the induced system satisfies the functional limit theorem, utilizing standard methods from \cite{durrett1978functional}, and then applying the lifting principle of \cite{melbourne2015weak}.

The rest of the paper is organized as follows. In the next section we state our main result. In Section \ref{sec:top-level proof} we give the top-level proof of our main result as just described in the previous paragraph. In Section \ref{sec:proof} we provide the details of the limit theorem for the induced system via the mechanism of   \cite{durrett1978functional}. In the appendix we present the additional technical details needed to handle the general setting of asymmetric cusps we consider here, and also provide  (for any reader who is seeing these for the first time) a very brief introduction to the Skorokhod $J_1$ and $M_1$ topologies.
\section{Main result}

A function $f$ with $\mu(f)=0$ is said to be in the domain of attraction of a (strictly) $\alpha$-stable law if there exists $\{b_n\}$ such that $\{\frac{\cS_nf}{b_n}\}$ converges in distribution to a random variable with an $\alpha$-stable law. Here, {\em strictly} simply means that $\mu(f)=0$, and we shall henceforth just say $\alpha$-stable.
 Contrary to central limit behavior and the Gaussian case of $\alpha=2$, it
 is well-known that even though we have a mean of zero, the limiting distribution may not be symmetric.
In particular, there exists a {\em skewness parameter} $\xi\in [-1,1]$ and $s>0$ such that
the limit stable random variable $\mathcal Y$ satisfies
\beq\label{tailstable}
\lim_{x\to\infty}x^\alpha\P(\mathcal Y>x)=C_\alpha s^\alpha\frac{1+\xi}{2} \text{ and }\lim_{x\to-\infty}|x|^\alpha\P(\mathcal Y<x)=C_\alpha s^\alpha\frac{1-\xi}{2}\, ,
\eeq
with $C_\alpha$ given by
$$C_\alpha=\frac{1}{\Gamma(1-\alpha)\cos(\pi\alpha/2)},$$
see \cite[p.17]{samorodnitsky1994stable}.
We will henceforth denote by $\sab$ a stable random variable
{with tail distribution satisfying \eqref{tailstable},} i.e. with characteristic function
\beq
\E\(e^{i u \mathcal Z_{{\alpha,\xi,s}}}\right)=\exp\(-
|us|^\alpha\(1-i\xi \text{sign}(u)\tan\frac{\pi\alpha}{2}\right)\right)\, ,\quad u\in\mathbb R.
\eeq
For any $\gamma\in (0,1)$, we denote $\cH_{\gamma}$ as the class of all H\"older continuous functions {$f:\mathcal M\setminus S_0\to\R$}, with H\"older exponent $\gamma$. 
Let us write $\tilde r_i$ for the curvilinear abscissa of the cusp position $P_i$. We define
$$\forall \varphi\in[0,\pi],\quad \tilde f_{i,+}(\varphi):=\lim_{r\rightarrow \tilde r_i+}f(r,\varphi) \quad\mbox{and}\quad \tilde f_{i,-}(\varphi):=\lim_{r\rightarrow \tilde r_i-}f(r,\varphi)\, ,$$
so that $\tilde f_{i,-}$ (resp. $\tilde f_{i,+}$) corresponds to the limit function on $P_i\times[0,\pi]$ of $f|_{\Gamma_i}$ (resp.  $f|_{\Gamma_{i+1}}$).
We also define
\beq
I_{f,i}=I_{f,i,\alpha}:=\frac 14\int_{0}^{\pi} (\tilde f_{i,-}(\varphi)+\tilde f_{i,+}(\varphi))\sin^{\frac{1}{\alpha}}\varphi\, d\varphi\, .
\eeq

\begin{theorem}[Stable limit theorem for billiard with cusps]\label{thm:1}
Let $Q$ be a  billiard table as described above, with cusps defined by (\ref{z1s}). Suppose $f\in \cH_{\gamma}$ for some $\gamma>0$ satisfying $\mu(f)=0$ and suppose there exists some $i\in\mathcal J_*$ such that $I_{f,i}\neq 0$.
Then as $n\to\infty$,
\begin{equation}\label{stablelaw}
\frac{\cS_{n} f}{n^{1/\alpha}}\xrightarrow{d} 
\sum_{i\in \mathcal J_*,\ I_{f,i}\neq 0}
 {\cal Z}_{\alpha,\xi_{f,i}, \frac{\sigma_{f,i}}{
 C_\alpha^{1/\alpha}
 }
 }\, ,
\end{equation}
where the limit is a  sum of independent   stable  variables with
$$\sigma_{f,i}^\alpha:=\frac{2}{|\partial Q|}\cdot \frac{I_{f,i}^\alpha}{\beta\bar C_i^{\alpha-1}}\,\,\,\,\,\,\text{ and }\,\,\,\,\,\, \xi_{f,i}:=\sign (I_{f,i}).$$
\end{theorem}


\begin{theorem}[Functional $\alpha$-stable limit theorem]\label{thm1}
Under the assumptions of Theorem \ref{thm:1} if, for every $i\in\mathcal J_*$, there exists a neighbourhood $\mathcal U_i$ of $P_i\times [0,\pi]$ such that either $f|_{\mathcal U_i}\ge 0$ or $f|_{\mathcal U_i}\le 0$,
then, in addition to the convergence in \eqref{stablelaw},
 $$((W_n(t))_{t\in[0,1]})_n,$$ defined in \eqref{defnXnt}, converges in distribution in the Skorokhod space $\mathbb D([0,1])$ using the
$M_1$-metric, to an  $\alpha$-stable L\'evy motion $(\mathcal Y_t)_{t\in[0,1]}$ (an $\alpha$-stable process with stationary and independent increments) such that
$\mathcal Y_1$ has the same distribution as the right side of \eqref{stablelaw}.
Moreover, this convergence does not hold for the $J_1$-metric.
\end{theorem}

Note that the values of $f$ at the maximally flat cusps determine the value $\xi$ and $\sigma$ and in fact one can compute that the right side of \eqref{stablelaw} is just a 
random variable $\mathcal Z_{\alpha,\xi_f,C_\alpha^{-\frac 1\alpha}\sigma_f}$ with
\beq\label{sigmaxi}
\sigma_f^\alpha:=\sum_{i\in\mathcal J_*}\frac{2|I_{f,i}|^\alpha }{\beta\bar C_i^{\alpha-1}|\partial Q|}\quad\mbox{and}\quad\xi_f:=\frac{\sum_{i\in \mathcal J_*}\sign(I_{f,i})\bar C_i^{1-\alpha}|I_{f,i}|^\alpha}{\sum_{i\in \mathcal J_*}\bar C_i^{1-\alpha}|I_{f,i}|^\alpha}\, .
\eeq
Let us remark that, as in \cite{JZ17}, Theorem \ref{thm1} (as well as Theorem \ref{thm:induced} below) holds also for $f$ bounded and piecewise $\gamma$-H\"older, for some $\gamma>0$, with discontinuities contained in the singular set of $T$ and such that $f$ is $\gamma$-H\"older in  a neighborhood of the region in $\cM$ corresponding to each of the cusps. 
%
%
\section{Proofs of Theorems \ref{thm:1} and \ref{thm1} using an induced map}\label{sec:top-level proof}

\subsection{Convergence for the $M_1$-metric}
Recall that $$F(x):=T^{\mathcal R(x)}(x),$$ with $\mathcal R(x):=\min\{n\ge 1\, :\, T^n(x)\in M\}$.
For every $g:M\rightarrow\mathbb R$, we denote the Birkhoff sums on the induced space by
$$S_ng(x):=\sum_{k=0}^{n-1}g\circ F^k(x),\quad x\in M,\ n\ge 0\, .$$
Recall also that, given $f:\mathcal M\rightarrow\mathbb R$, we define $\tilde f:M\rightarrow \mathbb R$ by setting
\begin{align}
\tilde f(x):=\mathcal S_{\mathcal R(x)}f(x).
\end{align}
We similarly define the process $\hW_n:=(\hW_n(t))_{{t\in[0,1]}}$ as follows:
\begin{align}
\hW_n(t) :=n^{-1/\alpha} S_{[nt]} \tf\, .
\end{align}
An intermediate functional limit theorem for the induced map is as follows.

\begin{theorem}\label{thm:induced}
Let $f:\cM\to \mathbb{R}$ be as in Theorem 
\ref{thm:1}.
Then $(\hW_n)_n$ converges in distribution, in the
Skorokhod $J_1$-topology, to an $\alpha$-stable L\'evy motion with $\alpha=\frac{\beta}{\beta-1}$, skewness parameter $\xi_f$, and scale parameter $$s=C_\alpha^{-1/\alpha}\tilde\sigma_f:=(C_\alpha\mu(M))^{-1/\alpha}\sigma_f\, ,$$
with $\sigma_f$ and $\xi_f$ defined in \eqref{sigmaxi}
\end{theorem}

Both Theorems \ref{thm:1} and \ref{thm1} are consequences of Theorem \ref{thm:induced}.
\begin{proof}[Proof of Theorem \ref{thm:1}]
By considering the first hitting time of the $M$ (in lieu of the first return time) we may extend the definitions of $\mathcal R$ and $F$ to all of $\mathcal M$, so that in particular $\tilde f\circ F$ is defined on all of $\cal M$.
Due to Theorem \ref{thm:induced}, $$\left((n^{-\frac 1\alpha}S_{\lfloor nt\rfloor}\tilde f\circ F)_{t\ge 0}\right)_n$$ converges in distribution, with respect to $\mu$ and to the metric $J_1$, to a L\'evy process $(\mathcal Y_t)_{t\ge 0}$. Moreover, for every positive integer $m$,
we define $\tau_m(x)$ as the number of visits of $(T^k(x))_{k=1,...,m}$ to $M$ so that
$$ \sum_{k=0}^{\tau_m(x)-1}\mathcal R\circ F^k(x)\le m<\sum_{k=0}^{\tau_{m}(x)}\mathcal R\circ F^k(x)\, .$$
Due to the ergodic theorem $(\tau_m/m)_{m\ge 1}$ converges almost surely to $\mu(M)$ and so
$\left((n^{-\frac 1\alpha}S_{\tau_{\lfloor nt\rfloor}-1}\tilde f\circ F)_{t\ge 0}\right)_n$ converges in distribution, with respect to $\mu$ and to $J_1$, to $(\mathcal Y_{t\mu(M)})_{t\ge 0}$.
Moreover,
$$\left|\mathcal S_{m}f(x)-S_{\tau_{m}(x)-1}\tilde f(F(x))\right|\le \Vert f\Vert_\infty\left(\mathcal R(x)+\mathcal R_-(T^m(x))\right) \, ,$$
where $\mathcal R_-$ corresponds to the length of time since the last visit to $M$.
Therefore by reversibility of the dynamics, for every $\epsilon>0$,
$$\mu\left(n^{-\frac 1\alpha}\left|\mathcal S_{\lfloor nt\rfloor}f-S_{\tau_{\lfloor nt\rfloor}(x)-1}\tilde f\circ F\right|>\epsilon\right)\le
2\,  \mu\left(n^{-\frac 1\alpha}\Vert f\Vert_\infty\mathcal R>\varepsilon/2\right)\, . $$
This implies the convergence in probability of $(n^{-\frac 1\alpha}(\mathcal S_{\lfloor nt\rfloor}f(x)-S_{\tau_{\lfloor nt\rfloor}(x)-1}\tilde f(F(x))))_n$ to 0 for every $t\ge 0$ and so the convergence of the finite-dimensional distributions (with respect to $\mu$) of $\left((n^{-\frac 1\alpha}\mathcal S_{\lfloor nt\rfloor}f)_{t\ge 0}\right)_n$ to the ones of $(\mathcal Y_{t\mu(M)})_{t\ge 0}$. 
\end{proof}

The result of convergence in distribution of
Theorem \ref{thm1} will follow directly from Theorem \ref{thm:induced} and Proposition \ref{lem:MZ} below. The latter is a version of \cite[Thm 2.2]{melbourne2015weak}, the statement of which, requires the following notion:
$$
f^*(x):=\(\max_{0\le\ell'\le\ell\le \cR(x)}(\cS_{\ell'}f(x)-\cS_{\ell}f(x))\)\wedge\(\max_{0\le\ell'\le\ell\le \cR(x)}(\cS_{\ell}f(x)-\cS_{\ell'}f(x))\).
$$

\begin{proposition}
[Lifting a weak invariance principle, Melbourne and Zweim\"uller]\label{lem:MZ}
Under the above assumptions on $f$ and $\hW_n$, if $(\hW_n(t), {t\in[0,1]})$ converges weakly in the
Skorokhod $M_1$-topology, as $n\to\infty$, to an $
\alpha$-stable L\'evy motion $(W(t), {t\in[0,1]})$ and
\beq\label{max_condition}
n^{-1/\alpha}\(\max_{0\le k\le n} f^*\circ F^k\)\xrightarrow{d} 0.
\eeq
then $(W_n(s), s\ge 0)\xrightarrow{d}  (W(s\mu(M)), s\ge 0)$ in the
Skorokhod $M_1$-topology.
\end{proposition}

Note that the version of this result in \cite{melbourne2015weak} assumes strong distributional convergence of the induced process, but Proposition 2.8 there allows us to state it as we did above.
\begin{proof}[Proof of Theorem \ref{thm1}]
We need only check the so-called {\it weak monotonicity} condition \eqref{max_condition} in our setting. Since $f$ takes a single sign in the neighborhood around each cusp 
$P_i$ with $i\in\mathcal J_*$,
we have nearly full monotonicity of the ergodic sums during an excursion, with the only exception possibly coming from the first step. Thus
	\begin{align*}
	f^*&\le\Vert f\Vert_\infty\mathcal R\bI_{\cup_{i\in\mathcal J\setminus \mathcal J_*}M_i}+\Vert f\Vert_\infty\bI_{
\cup_{i\in\mathcal J_*}M_i}\\
	&\le \Vert f\Vert_\infty(\mathcal R\bI_{\cup_{i\in\mathcal J\setminus \mathcal J_*}M_i}+1).
	\end{align*}
Now let $\tilde\alpha:=\min\{\frac{\beta_i}{\beta_i-1}:i\in\mathcal J\setminus \mathcal J_*\}$, which by assumption satisfies $\tilde\alpha>\alpha$. 
Due to Section \ref{sec:proof}, 
$$\left(Y_n(t):=n^{-1/\alpha} S_{[nt]} \(\mathcal R\bI_{\cup_{i\in\mathcal J\setminus \mathcal J_*}M_i}-\tmu(\mathcal R\bI_{\cup_{i\in\mathcal J\setminus \mathcal J_*}M_i})\)\right)_t$$
converges	in the $J_1$-metric, as $n$ goes to infinity, to the 0 function (here $\tilde\mu:=\mu(\cdot|M)$).
	Therefore
	\begin{eqnarray*}
		n^{-\frac 1\alpha}\(\max_{0\le k\le n} f^*\circ F^k\)&\le&
		\Vert f\Vert_\infty n^{-\frac 1\alpha}\max_{0\le k\le n}\( \left( \mathcal R\bI_{\cup_{i\in\mathcal J\setminus \mathcal J_*}M_i}-\tmu(\mathcal R\bI_{\cup_{i\in\mathcal J\setminus \mathcal J_*}M_i})\right)\circ F^k+1+\tilde\mu(\mathcal R)\)\\
		&\le&\Vert f\Vert_\infty \left(\left|\max_{t\in[0,1]}Y_n(t)-\min_{t\in[0,1]}Y_n(t)\right|+n^{-\frac 1\alpha}(1+\tilde\mu(\mathcal R))\right)\, ,
	\end{eqnarray*}
	which converges in distribution to 0.\\
 This checks that \eqref{max_condition} holds which completes the proof of Theorem \ref{thm1} up to a proof of Theorem \ref{thm:induced} which is given in
Section \ref{sec:proof}.
\end{proof}
\subsection{Non-convergence for the $J_1$-metric}\label{sec:j1}
In this subsection, we explain why the convergence in Theorem \ref{thm1} does not hold in the $J_1$-topology (see \cite[Example 2.1]{tyran2010weak} or \cite{avram1992weak} for other redactions of this argument).

Let $(w_n(t))_t$ be the continuous process obtained from
$(W_n(t))_t$ by linearization:
$$
w_n(t):=W_n(t)+\frac{(nt-\lfloor nt\rfloor)f\circ T^{\lfloor nt\rfloor}}{n^{1/\alpha}}\, ,\quad {t\in[0,1]}.
$$
Since $f$ is uniformly bounded, we also have
$$\sup_{{t\in[0,1]}} \left| w_n(t)-W_n(t)\right|\le \frac{\Vert f\Vert_\infty}{n^{1/\alpha}}\longrightarrow 0,\quad as\quad n\rightarrow \infty\, .$$
Therefore, the convergence in distribution of $(W_n(t))_n$
in the $J_1$-metric would imply the same convergence in distribution for $(w_n(t))_n$ in the $J_1$-metric and so for the uniform metric on every compact interval, this would contradict the fact that the limiting process is discontinuous.

Note that the reason this argument does not apply to the induced process in Theorem \ref{thm:induced} (or to $J_1$-convergence for general L\'evy processes), is because in that setting $\|\tilde f\|_{\infty}=\infty$.
\section{Proof of Theorem \ref{thm:induced} (functional limit theorem for the
induced system)}\label{sec:proof}

Assume the assumptions of Theorem \ref{thm:1}.
Recall that $\tilde\mu=\mu(\cdot|M)$.
We define $$\tf_i:=\tf \cdot \mathbf 1_{M_i}-\tilde\mu(\tf\cdot\mathbf 1_{M_i}).$$
Note that, due to the Kac theorem,
 	\begin{equation}\label{eq:finite mean}
 	\tmu(\cR)<\infty.
 	\end{equation}

\subsection{Simplification}
In order to prove Theorem \ref{thm:induced}, we show that we can replace $\tilde f$ by the simplest function $\tilde R$ given by:
\begin{align}
\tilde R(x):={\sum_{i\in\mathcal J} \frac{I_{f,i,\alpha_i}}{I_{1,\alpha_i}}(\mathcal R_i(x)-{\tilde\mu(\mathcal R_i)})\, ,\quad\mbox{with}\quad \mathcal R_i:=\mathcal R\mathbf 1_{M_i}}
\end{align}
with
$$I_{1,\alpha_i}:=I_{1,i,\alpha_i}=\int_{0}^{\pi/2}\sin^{\frac{1}{\alpha_i}}\varphi\,d\varphi\quad\mbox{and}\quad \alpha_i=\frac{\beta_i^*}{\beta_i^*-1}\, .$$
We simply write $I_1$ for $I_{1,\alpha}$.
The goal of this section is to prove the following result.
\begin{proposition}\label{PROPsimpl}
We have
$$
\sup_{k\le n} \frac{S_k(\tilde f-\tilde R)
}{n^{1/\alpha}}
\rightarrow 0\, ,$$
in probability, as $n$ goes to infinity.
\end{proposition}
This proposition will come from the following lemmas which require some notations.
Fix a small $\delta>0$, and split $M_i$ according to the low, intermediate, and high regions of the index $m$
for the sets \mbox{$M_{i,m}:=\{x\in M_i\,:\, \cR(x)=m\}$:}
\beq
M_i^L:=\cup_{m<\delta n^{\frac{1}{\alpha}}} M_{i,m},\quad\quad M_i^I:=\cup_{\delta n^{\frac{1}{\alpha}}\le m< \frac{1}{\delta}n^{\frac{1}{\alpha}}} M_{i,m} \quad\text{and}\quad M_i^H
:=\cup_{m\geq \frac{1}{\delta}n^{\frac{1}{\alpha}}} M_{i,m}
\eeq
which all depend implicitly on $n$ and $\delta$. We also denote
$$\tf_{i,\zeta}:=\tf\cdot\mathbf 1_{M_i^{\zeta}}-\tilde\mu(\tf\cdot\mathbf 1_{M_i^{\zeta}}),$$ with $\zeta\in\{L,{I},H\}$.

Similarly to \cite[Lemma 4.3]{JZ17}, $\tf_i$ is essentially determined by $\tf_{i,I}$. Indeed, by considering the  induced system restricted to each $M_i$, {the proof of} Lemma 4.3 in \cite{JZ17} gives us the following result:
\begin{lemma}[Vanishing of truncated portions]\label{LHzero}
	For every {$i\in\mathcal J$}, we have, in probability,
	\begin{align}\label{bound on H}
	\lim_{\delta\to 0}\sup_{k\le n}\frac{S_k\mathbf 1_{M_i^H}}{n^{\frac{1}{\alpha_i}}}=0\, .
	\end{align}
Moreover, there exist $C'_0>0$, $\theta\in(0,1)$ and $\kappa>0$ such that, for every {$i\in\mathcal J_*$} and every $k\ge 1$, we have
\beq\label{decay cov}\cov\left(\tf_{i,L},\tf_{i,L}\circ F^k\right)\le C'_0\delta^{\kappa}\theta^kn^{\frac 2{\alpha_i}-1-\kappa}\, . \eeq
\end{lemma}
For $\zeta\in\{L,I,H\}$, set $\cR_{i,\zeta}:=\cR|_{M_i^\zeta}$ and define $E_{i,\zeta}$ as:
\beq\label{defnIfI}
E_{i,\zeta}:=\tilde{f}_{i,\zeta} -\frac{I_{f,i}}{I_1}\(\cR_{i,\zeta}(x)-\tmu(\cR_{i,\zeta}
)\)\, .
\eeq
Again we note that the truncation to the intermediate region, denoted by $I$ in the subscripts/superscript, implicitly depends on $n$. Similar to \cite{JZ17}, Lemma 4.4 and Proposition 4.2 in \cite{JZ17} implies that  $f$ is $\gamma$-H\"older as the following lemma indicates. 


\begin{lemma}\label{lem2}
There exist $C=C(\delta)>0$ and $\vartheta\in(0,1)$ such that
\beq\nn
|E_{i,I}|\leq C \|f\|_{\ff}|\cR_{i,I}-\tmu(\cR_{i,I})|^{1-\frac{\gamma}{\beta_i^*-1} } \quad \mbox{on }{M_{i}^I}
\eeq
and
$$\cov\left(E_{i,I},E_{i,I}\circ F^k\right) \le C
n^{\frac 2{\alpha_i}-1-\frac 1{\alpha_i(\alpha_i+1)}}\vartheta^k $$
\end{lemma}
The proof of the above lemma is similar to arguments in \cite{JZ17}; we indicate which estimates in the proof need updates and/or adjustments at the end of Appendix \ref{sec:scale}.

\begin{proof}[Proof of Proposition \ref{PROPsimpl}]
Lemma \ref{lem2} implies that there exists $K>0$ such that for all $m$,
$${\tilde\mu ((S_mE_{i,I})^2)=}\var(S_mE_{i,I})\le 2\sum_{k=0}^{m-1}(m-k) \cov(E_{i,I},E_{i,I}\circ F^k) \le  K\,
m\, n^{\frac 2{\alpha_i}-1-\frac 1{\alpha_i(\alpha_i+1)}}$$
The exponent of $n$ can be rewritten $-\frac{\alpha_i^2-\alpha_i-1}{\alpha_i(\alpha_i+1)}\le -\frac{\alpha^2-\alpha-1}{\alpha(\alpha+1)}$
since $\alpha_i\ge \alpha$.
So, by \cite[Exercise 12.5]{billingsley1968convergence} (with $\gamma=2,\alpha=1$ there, see also \cite[Eq. (2.3)]{serfling1970moment}), for all $m$,
$$\tilde\mu\left(\sup_{k=1,...,m}|S_kE_{i,I}
|^2\right)\le (\log_2 4m)^2 K\, 
m\, n^{-\frac{\alpha^2-\alpha-1}{\alpha(\alpha+1)}}.$$
there exist $K_0=K_0(\delta)>0$ and $\varepsilon>0$ such that,
\begin{equation}\label{majosup}
\text{for all } n\ge 0,\quad \tilde\mu\left(\sup_{k=1,...,n}|S_kE_{i,I}
|^2\right)\le  K_0\,  n^{\frac 2\alpha-\varepsilon}\, .
\end{equation}
Proceeding analogously, due to Lemma \ref{LHzero} applied to $\tilde f-\tilde R$ instead of $\tf$, there
exist $K'_0>0$ and $\kappa>0$ such that
$$
\text{for all } n\ge 0,\quad \tilde\mu\left(\sup_{k=1,...,n}
\left|S_k E_{i,L}
\right|^2\right)\le  K'_0\, \delta^{\kappa} n^{\frac 2\alpha-\varepsilon}\, .
$$

Fix $\varepsilon>0$.
{Due to the Markov inequality}, we can find $\delta>0$ small enough so that for all $n$
$$\tilde\mu\left(\sup_{k\le n}\frac{\left|S_k
\left(\sum_{i\in\mathcal J}E_{i,L} \right)
\right|
}{n^{\frac 1\alpha}}>\frac\varepsilon 4\right)<\frac\varepsilon 4\,  .$$
Also, by \eqref{bound on H}, we can choose $\delta>0$ small enough so that for all $n$
	$$\tilde\mu\left(\sup_{k\le n}\frac{\left|S_k
			\left(\sum_{i\in\mathcal J}E_{i,H} \right)
			\right|
	}{n^{\frac 1\alpha}}>\frac\varepsilon 4\right)<\frac\varepsilon 4\,  .$$
Now, due to \eqref{majosup} and Markov's inequality,
for  the minimum of these choices of $\delta$, and for $n$ large enough,
$$\tilde\mu\left(
    \sup_{k\le n}\frac{\left|S_k\left(\sum_{i\in\mathcal J} E_{i,I}\right)\right|}{n^{1/\alpha}}
     >\frac\varepsilon 2 \right)<\frac\varepsilon 2.$$
\end{proof}

%
%
%
%
\subsection{Convergence in distribution}
Due to the simplification in the previous subsection, it remains now to prove the convergence in distribution as $n\to\infty$ for the following functions of $t$ (in the $J_1$-metric)
$$\left(\left(\mathcal Z_n(t):=\frac{S_{[nt]}\sum_{i\in\mathcal J} A_i\left(\mathcal R_i-\tilde\mu(\mathcal R_i)\right)}{n^{1/\alpha}}\right)_{t}\right)_n$$
for any fixed real numbers {$\{A_i, i\in\mathcal J_*\}$.}
To this end, we consider the family of point processes $(N_n)_n$
on 
$(0,+\infty)\times(\mathbb R\setminus\{0\})$
given by
$$N_n:=\sum_{j= 1}^n\delta_{\left(\frac j n,\frac{Z_{j-1}}{n^{1/\alpha}}\right)},\quad\mbox{with}\ Z_j:=\left[\sum_{i\in\mathcal J} A_i (\mathcal R_i-\tilde\mu(\mathcal R_i))\right]\circ F^j\, , $$
and we employ standard methods for proving functional limit theorems from  \cite[Sec. 4]{durrett1978functional}. Here we state a version from \cite[Thm 1.2]{tyran2010weak}.
\begin{proposition}\label{prop:tk}
Assume the following two conditions.

\noindent\textbf{Condition I.} (Point process convergence).
The sequence of point processes $(N_n)_n$ converges in distribution
to a Poisson point process $N$ on 
$(0,+\infty)\times (\mathbb R\setminus\{0\})$  with intensity 
$\gamma$ having density $\psi$ with respect to the Lebesgue
measure on $(0,+\infty)\times(\mathbb R\setminus\{0\})$, with
$$\psi(t,y)= \sum_{i\in\mathcal J_*}  \frac{2
I_1^\alpha\, |A_i|^\alpha}{\beta\bar C_i^{\alpha-1}\mu(M)|\partial Q|}{\mathbf 1_{\{y A_i>0\}}}\alpha|y|^{-\alpha-1}.$$
\noindent\textbf{Condition II} (Vanishing small values). For all $\gamma>0$
$$\lim_{\eps\to 0}\limsup_{n\to\infty}\tmu\left[\max_{0\leq k\leq n}\left|\sum_{j=0}^{k-1}\left(\frac{Z_{j}}{n^{1/\alpha}}\cdot\bI_{|Z_{j}|/n^{1/\alpha}<\eps}\right)-k\tmu\left(\frac{Z_{1}}{n^{1/\alpha}}\cdot\bI_{|Z_{1}|/n^{1/\alpha}<\eps}\right)\right|>\gamma\right]=0.$$

Then $((\mathcal Z_n(t))_t)_n$ (and so $((\frac1{n^{1/\alpha}}S_{\lfloor nt\rfloor} \tf)_t)_n$) converges in distribution
(in the $J_1$-metric) to an $\alpha$-stable L\'evy motion $(\mathcal Y_t)_t$ such that $\mathcal Y_1$ has the same distribution as
$ Z_{\alpha,\xi,\sigma}$
with
$$\sigma^\alpha:=\sum_{i\in\mathcal J_*}  \frac{2I_1^\alpha\, |A_i|^\alpha}{\beta\bar C_i^{\alpha-1}\mu(M)|\partial Q|}\,\,\,\,\,\,\text{ and }\,\,\,\,\,\, \xi:=\frac{\sum_{i\in\mathcal J_*} \sign(A_i)\,  \frac{2I_1^\alpha\, |A_i|^\alpha}{\beta\bar C_i^{\alpha-1}\mu(M)|\partial Q|}}{\sum_{i\in\mathcal J_*}  \frac{2I_1^\alpha\, |A_i|^\alpha}{\beta\bar C_i^{\alpha-1}\mu(M)|\partial Q|}}.$$
\end{proposition}

\subsection{A key estimate for Condition~I}
In order to prove the convergence in distribution $N_n\rightarrow N$, we use the Kallenberg theorem or a method based on this theorem.
A key result in this study is the following lemma, which was proved as  Lemma 2.2 \cite{JZ17} only for one specific symmetric cusp.  The generalization of this in our context is the key estimate that shows that only cusps of maximal flatness are relevant. The modifications to the proof given in \cite{JZ17} is nontrivial, thus we provide the proof of this lemma in Appendix \ref{sec:scale}.
\begin{lemma}\label{Mm2}
The return time function satisfies 
\beq
	\forall i\in\mathcal J,
	\,\,\,\,\,\,\,\label{eq:scale parameter}\lim_{y\to{+\infty}} y^{\alpha_i}\tilde\mu\left(x\in M_i: \cR(x) >y\right)=\frac{2
		I_1^{\alpha_i}}{\beta_i^*\bar C_i^{\alpha_i-1}\mu(M)|\partial Q|}\, 
,\eeq
with $\alpha_i:=\frac{\beta_i^*}{\beta_i^*-1}$.
\end{lemma}
Since $\alpha_i>\alpha$ for $i\notin {\cal J_*}$, this lemma ensures in particular that
$$ \lim_{y\to{+\infty}} y^{\alpha}\tilde\mu\left(x\in M_i: \cR(x) >y\right)>0\quad \Leftrightarrow\quad i\in\mathcal J_*\, .$$
\begin{corollary}\label{coroMm2}
For any $\epsilon>0$,
\begin{align*}
\lim_{y\to+\infty} y^{\alpha} \tilde\mu( \epsilon^{1/\alpha}Z_{-1}>y)
&=\lim_{y\to+\infty} y^{\alpha}\sum_{i\in\mathcal J}
\tilde \mu\left(M_i\cap\{ A_i \mathcal R_i-\sum_jA_j\tilde\mu(\mathcal R_j)>\epsilon^{-\frac 1\alpha}y\}\right)\\
&=\lim_{y\to+\infty} y^{\alpha}\sum_{i\in\mathcal J\, :\, A_i>0}
\tilde \mu\left(M_i\cap\left\{\mathcal R>\frac{\epsilon^{-\frac 1\alpha}y}{A_i}+\frac 1{A_i}\sum_jA_j\tilde\mu(\mathcal R_j)\right\}\right)\\
&= \epsilon \sum_{i\in\mathcal J_*\, :\, A_i>0}  \frac{2
	I_1^\alpha\, A_i^\alpha}{\beta\bar C_i^{\alpha-1}\mu(M)|\partial Q|}\, .
\end{align*}
and analogously
$$\lim_{y\to+\infty} y^{\alpha} \tilde\mu( \epsilon^{1/\alpha}Z_{-1}<-y)
= \epsilon \sum_{i\in\mathcal J_*\, :\, A_i<0}  \frac{2
	I_1^\alpha\, |A_i|^\alpha}{\beta\bar C_i^{\alpha-1}\mu(M)|\partial Q|}\, .
$$
\end{corollary}
\subsection{Condition~I: Point process convergence}
In order to prove the convergence in distribution $N_n\rightarrow N$, due to the Kallenberg theorem \cite{kallenberg1973characterization} (see also \cite[Prop. 3.22]{resnick1987extreme}), it is enough to prove that
\begin{equation}\label{Kall1}
\lim_{n\rightarrow +\infty}\tilde\mu\left(N_n(R)\right)=\eta(R)\,
\end{equation}
and
\begin{equation}\label{Kall2}
\lim_{n\rightarrow+\infty}\tilde\mu(N_n(R)=0)=e^{-\eta(R)}\, ,
\end{equation}
for every $R$ of the form $$R=\bigcup_{i=1}^m(a_i,b_i)\times 
I_{c_i,c'_i}
,$$
with $I_{c,c'}=(-\infty,-c)\cup(c',+\infty)$,
$0<a_i<b_i<1$ and $c_i,c'_i>0$ for every $i$.
In $\R\backslash\{0\}$, we fix a subset 
{$\cI=I_{c,c'}$ with $c,c'>0$.}
Our correlation bounds will depend on sets of the form
$$D_{n,j}:=\{{x\in M}:\frac{1}{\sqrt[\alpha]{n}}(\cR_j-c)\circ F^{j}(x) \in \cI\}.$$
     \begin{lemma}[Exponential decay of correlations for $q$-point marginals, \cite{CZ09, JZ17}]\label{boundedcov}
For every $\cI$, there is a constant $C>0$  and $\theta\in(0,1)$ such that 
\beq\label{eq:condition D_r(u_n)}
\tmu({D}_{n,1}^c\cap{\cdots}\cap {D}_{n,q}^c\cap{D}_{n,q+k+1}^c\cap\cdots\cap{D}_{n,2q+k}^c)-\(\tmu({D}_{n,1}^c\cap\cdots\cap {D}_{n,q}^c)\)^2\leq C \theta^k
\eeq
for all $k,n,q\in\N$ satisfying $2q+k\le n$.
Also, there exists $\theta_0>0$ such that for all $1\le i<j\le n$
\beq\label{eq:condition D_r(u_n)2}
\tmu({D}_{n,i}\cap{D}_{n,j})\leq o\(\frac1{n^{1+\theta_0}}\).
\eeq
     \end{lemma}
 
 The first part above, \eqref{eq:condition D_r(u_n)} follows from Theorem 4 in \cite{CZ09} which covers the setting we are in. For \eqref{eq:condition D_r(u_n)2}, the proof of Lemma 3.2 in \cite{JZ17} can be used by replacing Lemma 2.1 there with the estimates we provide in Proposition \label{prop3} in the appendix.
\subsubsection{Proof of \eqref{Kall1}}
Let $m\ge 1$, and real 
{numbers $a_1,b_1,c_1,c'_1,...,a_m,b_m,c_m,c'_m$ be such that, for every $i$, $0<a_i<b_i$ and $c_i<c'_i$ with $c_ic'_i>0$}.
We assume without any loss of generality that $b_i\le a_{i+1}$.

Let $R:=\bigcup_{i=1}^m(a_i,b_i)\times 
I_{c_i,c'_i}
$, so that
$$\tilde\mu\left(N_n(R)\right) =\sum_{i=1}^m \sum_{j:na_i<j<nb_i}\sum_{k\in\mathcal J}\tilde\mu\left(M_k;\ 
({A_k}\mathcal R_k-\sum_{i}A_i\tilde\mu(\mathcal R_i))\circ F^j
 \in I_{c_i n^{\frac 1\alpha},c'_i n^{\frac 1\alpha}}
\right)\, .$$
Due to Corollary~\ref{coroMm2},
$$\tilde\mu\left(N_n(R)\right) \sim \frac{
2
 I_1^\alpha}{\beta\bar C_i^{\alpha-1}\mu(M)|\partial Q|}\sum_{i=1}^m \sum_{k\in\mathcal J_*}(b_i-a_i)\mathbf 1_{\{A_k c_i>0\}}|A_k|^\alpha.$$

Therefore, we have proved \eqref{Kall1} with $\eta$ of
densityas in Proposition \ref{prop:tk}.
%
%
%
%
\subsubsection{Proof of \eqref{Kall2}.}
In order to ease the notation below, we will prove only the case where $m=1$, $a=0$ and $b=1$. In this special case we can consider the canonical projection of $N_n$ onto its second argument, i.e., onto an empirical measure on $\R\backslash\{0\}$. Let us call this projection $\hat N_n$, and we replace the set $R$ with $\cI$ as defined above Lemma \ref{boundedcov}. 

In the general case, one simply needs to write $R$ in terms of a union of disjoint rectangles, and consider each of rectangles in the union separately; also the projection onto the second coordinate for each of these rectangles must be scaled by the Lebesgue measure of the time interval.

As in \cite[Prop. 3.4]{JZ17} we use Bernstein's block method \cite{bernstein1927extension}.
Condition \eqref{Kall2} would follow if the random variables $\{Z_j\}$ were independent since then we have for $Z_j=\left[(\mathcal R-\tilde\mu(\mathcal R))\sum_{i\in\mathcal J} A_i \mathbf 1_{M_i}\right]\circ F^j$, the splitting into a product measure
$$ \tilde\mu(\hat N_n(\cI)=0)=\prod_{j=1}^n\tilde \mu\left(n^{-1/\aa}Z_j\not\in \cI \right)$$
which easily implies the required convergence.
Since we do not have independence, we instead use asymptotic independence together with Bernstein's block method  to
control the dependence between the $\{Z_j\}$. {Choose $$0<v<w<\theta_0,$$ where $\theta_0$ is as in \eqref{eq:condition D_r(u_n)2}, and for any $n\geq 1$, divide $\{1,\ldots,n\}$ into a sequence of pairs of alternating big intervals (blocks) of length $[n^w]$ and
small blocks of length $[n^v]$. The number of
pairs of big and small blocks is ${B}= [n/([n^v] + [n^w])]$ so that  $\lim_{n\to\ff} B/n^{1-w}=1$. There may be a leftover partial block $L$ in the end which is negligible since
$$\sum_{j\in L} \tmu(n^{-1/\aa}Z_j\in \cI)\le \frac{C n^w}{n}.$$
Thus we may henceforth assume $n=([n^v]+[n^w]){B}$.}

We denote by $\cB_{k}$ and $\cS_{k}$ for $k=1,\cdots, {B}$, the elements of $\{1,\ldots,n\}$ in $k$th big block and small block, respectively.
Let
$$Y_{n,k}=\sum_{j\in \cB_{k}} \bI_{\{n^{-1/\aa}Z_j\in \cI\}},\,\,\,\,\,\,\,V_{n,k}=\sum_{j\in \cS_{k}} \bI_{\{n^{-1/\aa}Z_j\in \cI\}}.$$
For each $n$, both $\{Y_{n,k}\}$ and $\{V_{n,k}\}$ are sequences of identically distributed random variables. Let $S'_n=\sum_{k=1}^{B} Y_{n,k}$ and $S''_n=\sum_{k=1}^{B} V_{n,k}$.

Similar to the argument for $L$, we have
\begin{equation}
\tmu(S''_n) \le C\frac{{B} n^v}{n}\le C \frac{1}{n^{w-v}},
\end{equation}
so that we can ignore small blocks. Thus, it is enough to show that $\tmu(S'_n=0)$ converges to $\exp(-\lambda(\cI))$.
Since big blocks are separated by small blocks, we can peel off one factor at a time in the product $\prod \bI_{\{Y_{n,k}=0\}}$ as follows:
\begin{align*}
\tmu(S'_n=0)&=\tmu\(Y_{n,1}=0,Y_{n,2}=0,\ldots,Y_{n,B}=0\)\\
&=\tmu\(\prod_{k=1}^B\bI_{Y_{n,k}=0}\)\\
&\le\tmu\(\prod_{k=1}^{B-1}\bI_{Y_{n,k}=0}\)\tmu(\bI_{Y_{n,B}=0})+C\theta^{n^v}\\
&\le\(\tmu\(\prod_{k=1}^{B-2}\bI_{Y_{n,k}=0}\)\tmu(\bI_{Y_{n,B-1}=0})+C\theta^{n^v}\)\cdot\tmu(\bI_{Y_{n,B}=0})+C\theta^{n^v}.
\end{align*}
where we used \eqref{eq:condition D_r(u_n)} in the last two lines. Repeating this, we obtain for some $\theta_1\in (\theta,1)$,
\beq\label{Eexpo2}\tmu(S'_n=0)=(\tmu(Y_{n,1}=0))^{B}+\cO(\theta_1^{{n^v}}).\eeq
 It remains to estimate $\tmu(Y_{n,1}=0)$.
\begin{align*}
\tmu(Y_{n,1}=0)
&\le 1-\sum_{j=1}^{n^w}\left(\tmu(n^{-1/\aa}Z_j\in \cI)-\sum_{k=1}^{n^w}\tmu(\{n^{-1/\aa}Z_j\in \cI\}\cap\{n^{-1/\aa}Z_k\in \cI\})\right)\\
&\le 1-\sum_{j=1}^{n^w}\(\tmu(n^{-1/\aa}Z_j\in \cI)-o\(\frac1{n}\)\)\\
&\le 1-n^w\tmu(n^{-1/\aa}Z_1\in \cI)+n^w\,\,\, o\(\frac1{n}\)
\end{align*}
where the second inequality follows from \eqref{eq:condition D_r(u_n)2} since $w<\theta_0$.
An even easier lower bound is given by
\begin{align*}
\tmu(Y_{n,1}=0)
\ge 1-\sum_{j=1}^{n^w}\tmu(n^{-1/\aa}Z_j\in \cI).
\end{align*}

Putting things together we have
$$\tmu(S'_n=0)=\(1-{[n^w]}\tmu(n^{-1/\aa}Z_{1}\in \cI)+o(n^{w-1})\)^{B}+\cO(\theta_1^{n^v}),$$
which is what we need since $B/n^{1-w}\to 1$ and $n\tmu(n^{-1/\aa}Z_{1}\in \cI)\to \lambda(\cI)$.
\subsection{Condition I: alternative approach}
It is enough to prove the convergence of $(N_n)_n$ to $N$ as
a point process on $(0,+\infty)\times (\mathbb R\setminus[-a,a])$
for every $a>0$.
We fix $a>0$ and set $A_\epsilon:=\{|Z_{-1}|>a \epsilon^{-1/\alpha}\}$. Due to Corollary \ref{coroMm2}, $\tilde\mu(A_\epsilon)\sim \epsilon\eta_0(\mathbb R\setminus[-a,a])$ with
$\eta_0 $ the measure on $\mathbb R$ with density $\psi(1,\cdot)$ with respect to the Lebesgue measure. 
The convergence of $(N_n)_n$ to $N$ on $(0,+\infty)\times (\mathbb R\setminus[-a,a])$
follows directly from the following result.
\begin{lemma}\label{LEM0}
For every $a>0$, the family of point processes
$\left(\sum_{j\ge 1\ :\ T^jx\in A_\epsilon} \delta_{\left(j\tilde\mu(A_\epsilon),\epsilon^{1/\alpha}{Z_{j-1}}\right)}\right)_\epsilon$
converges in distribution, as $\varepsilon\rightarrow 0$,
to a Poisson process of intensity $\gamma/\eta_0(\mathbb R\setminus[-a,a])$
with respect to the Lebesgue measure on $(0,+\infty)\times (\mathbb R\setminus[-a,a])$.
\end{lemma}
\begin{proof}
We will apply \cite[Theorem 2.1]{FPBSPoisson2} (which uses the Kallenberg theorem) with
$A_\epsilon$ as above, with $H_\eps:A_\epsilon\rightarrow V:= \mathbb R\setminus[-a,a]$
given by $H_\epsilon(x):= \epsilon^{1/\alpha}Z_{-1}(x)$, with
$m:=\frac{\eta_0(\cdot \setminus[-a,a])}{\eta_0(\mathbb R\setminus[-a,a])}$
and with $$\mathcal W:=\{(c,c'); a<c<c'<\infty\}\cup \{(-c',-c); a<c<c'<\infty\}.$$
To apply \cite[Theorem 2.1]{FPBSPoisson2}, we have to prove first that
$\tilde \mu(H_\epsilon^{-1}(\cdot)|A_\epsilon)$
converges in distribution to $m$ (this is ensured by Corollary  \ref{coroMm2}) and second that,  for every $K\ge 1$ and every choice of intervals 
$W_1,...,W_K\in\mathcal W  $, 
\begin{equation}\label{TOTO0}
\Delta_{\epsilon,1}=o(\mu(A_\epsilon))\quad \mbox{(i.e. in }o(\epsilon)\mbox{)},
\end{equation}
with
$$\Delta_{\epsilon,n}:=\sup_{A,B\subset A_\epsilon\, :\, A\in\mathcal G_\epsilon,\, B\in\sigma\left(\bigcup_{j\ge n}F^{-j}(\mathcal G_\epsilon)\right)}\left|\tilde\mu(B\cap A)-\tilde\mu(A)\tilde\mu(B)\right| ,$$
and with $$\mathcal G_\epsilon:=\mathcal G_{\epsilon,W_1,...,W_K}:=\{H_\epsilon^{-1}(W_i); \ i=1,...,K\}.$$
As in the proof of \cite[Proposition 4.2]{FPBSPoisson2}, the fact that $\Delta_{\epsilon,1}=o(\epsilon)$ will follow from the following lemmas.\\
Set $$\tau_{A_\epsilon}:=\min\{n\ge 1\,:\, F^n(\cdot)\in A_\epsilon\}.$$
We will use a general argument given by the next general lemma and its general corollary.
\begin{lemma}[General result]\label{TOTO1}
For any positive integer $p$ and any $\epsilon>0$,
$$\Delta_{\epsilon,1}\le \Delta_{\epsilon,p+1} 
   +\tilde\mu(A_\epsilon)\left(\tilde\mu(\tau_{A_\epsilon}\le p|A_\epsilon)+\tilde\mu(\tau_{A_\epsilon}\le p)\right)\, .$$
\end{lemma}
\begin{proof}
Let $A\in\mathcal G_\epsilon$ and $B\in\sigma\left(\bigcup_{j\ge 1}H_\epsilon^{-1}(\mathcal G_\epsilon)\right)$. Note that there exists a function $g:(\{0,1\}^{K})^{\mathbb N}\rightarrow \{0,1\}$ such that $\mathbf 1_B=g(X_1,...)$, where
$X_i=\left(\mathbf 1_{H_\epsilon^{-1}(W_j))}\right)_{j=1,...,K}\circ F^i$.
Set $C$ such that
\begin{equation}\label{AAA0}
\mathbf 1_C:=g(\mathbf 0,...,\mathbf 0,X_{p+1},...)\in \sigma\left(\bigcup_{j\ge p+1}F^{-j}(\mathcal G_\epsilon)\right)
\end{equation}
and observe that $|\mathbf 1_B-\mathbf 1_C|\le {\mathbf 1}_{\{\tau_{A_\epsilon}\le p\}} $
so that
\begin{eqnarray}
\left|\cov(\mathbf 1_A,\mathbf 1_B)-\cov(\mathbf 1_A,\mathbf 1_C)\right|&\le \tilde\mu(A,\ \tau_{A_\epsilon}\le p)+\tilde\mu(A)\tilde\mu(\tau_{A_\epsilon}\le p)\, .\label{AAA}
\end{eqnarray}
Finally, due to \eqref{AAA0},
$|\cov(\mathbf 1_A,\mathbf 1_C)|\le\Delta_{\epsilon,p_\epsilon+1}\, .$
This combined with \eqref{AAA} and $A\subset A_\epsilon$ ends the proof of the lemma.
\end{proof}
\begin{corollary}[General result when $\mu(A_\epsilon)=\cO(\epsilon)$]\label{CORO0}
If there exists a family of positive integer $(p_\epsilon)_\epsilon$ such that $p_\epsilon=o(\epsilon^{-1})$ and
$\tilde\mu\left(\tau_{ A_\epsilon}\le p_\epsilon| A_\epsilon\right)=o(1)$, then
$$\Delta_{\epsilon,1}\le \Delta_{\epsilon,p_\epsilon+1} 
   +o(\epsilon)\, .$$
\end{corollary}
\begin{proof}
This corollary comes directly from Lemma \ref{TOTO1} combined with $\mu(A_\epsilon)=\cO(\epsilon)$ and
$$ \tilde\mu(\tau_{A_\epsilon}\le p_\epsilon)=\tilde\mu\left(\bigcup_{k=1}^{p_\epsilon}F^{-k}(A_\epsilon)\right)\le\sum_{k=1}^{p_\epsilon}\tilde\mu(F^{-k}(A_\epsilon))=p_\epsilon\tilde\mu(A_\epsilon)=o(1)\, .$$
\end{proof}
We fix $\theta\in(0,\theta_0)$
with $\theta_0=\varepsilon$ as in \eqref{eq:condition D_r(u_n)2} and set $p_\epsilon=\epsilon^{-\theta}$ so that $p_\epsilon\ll \epsilon^{-1}$ and
\begin{lemma}\label{lem:smallreturntimes}
We have
$\tilde\mu\left(\tau_{ A_\epsilon}\le p_\epsilon| A_\epsilon\right)=o(1) \, .$
\end{lemma}
\begin{proof}
Due to \eqref{eq:condition D_r(u_n)2},
$$\tilde\mu\left(\tau_{A_\epsilon}\le p_\epsilon| A_\epsilon\right)\le\sum_{k=1}^{p_\epsilon} \frac{\tilde\mu\left(A_\epsilon\cap F^{-k}A_\epsilon\right)}{\tilde\mu(A_\epsilon)}=O\left(p_\epsilon \frac{\epsilon^{1+\theta_0}}{\epsilon}\right)=O\left(\epsilon^{-\theta+ \theta_0}\right)=o(1)\, .$$
\end{proof}
\begin{lemma}[Decorrelation]\label{lemKall2:decorrelation}
$\Delta_{\epsilon,p_\epsilon+1} =o(\epsilon)\, .$
\end{lemma}
\begin{proof}
Let $A\in\mathcal G_\epsilon$ and $B\in\sigma\left(\bigcup_{j\ge p_\epsilon+1}F^{-j}\mathcal G_\epsilon\right)$.
Observe that $ A\in\mathcal G_\epsilon$ is a finite union of level sets of $\cR\circ F^{-1}$ intersected by some $F(M_i)$. Therefore $A$ can be smoothly foliated by a union of  unstable curves.
Moreover, since $F^{-j}(\mathcal G_\epsilon)$ can be smoothly foliated by  stable curves, the set $B$ can be rewritten 
$B=F^{-p_\epsilon}B'$ with $B'$ smoothly foliated by a union of  stable curves.
Therefore, due to \cite[Theorem 3]{CZ09}, there exists $z\in(0,1)$
and $C_0>0$ such that 
$$\Delta_{\epsilon,p_\epsilon+1} \le C_{0} z^{p_\epsilon}.$$
Note that  $C_0$ does not depend on $A$ and $B$, i.e., it is a uniform constant.
\end{proof}
Corollary \ref{CORO0} combined with Lemmas \ref{lem:smallreturntimes} and \ref{lemKall2:decorrelation} ensures \eqref{TOTO0} (for every $a>0$, for every $K\ge 1$ and every
$W_1,...,W_K\in\mathcal W $) and so Lemma \ref{LEM0}.
\end{proof}

\subsection{Condition II: Vanishing small values}
Recall that
$$Z_i:=\left[\sum_{j\in\mathcal J} A_j (\mathcal R_j-\tilde\mu(\mathcal R_j))\right]\circ F^i.$$
Condition II will follow immediately from \cite[Thm 10.1]{billingsley1999convergence} together with Lemma \ref{cond2lemma} below.
Given $\epsilon\in(0,1)$ and $n\ge 1$, we set
$$U_m(x)=U_{m,\epsilon,n}(x):=\sum_{i=0}^{m-1}\left(V_i(x)-\tilde\mu(V_0)\right),\quad\mbox{with}\quad V_i=V_{i,\epsilon,n}:=\frac{Z_{i}(x)}{n^{1/\alpha}}\cdot\bI_{\{|Z_{i}|/n^{1/\alpha}<\eps\}}.$$ 
We will also set $W_i:=V_i-\tilde\mu(V_i)$.
In order to prove Lemma \ref{cond2lemma}, we need an estimate of the variance of $U_m$.
\begin{lemma}\label{lem:VarUn}
There exists $\tilde C>0$ and $\theta_2>0$ such that,
for every $\epsilon\in(0,1)$ and for all $n$ large enough (more precisely, there exists $ n_\epsilon$ such that for all $n>n_\epsilon$),
$$\sum_{k\ge 0} \left|\cov_{\tilde\mu}(V_0,V_k)\right|
    =\sum_{k\ge 0} \left|{\tilde\mu}(W_0W_k)\right|\le \tilde C\left(\epsilon^{2-\alpha}n^{-1}+\epsilon^{\theta_2}n^{-1-\theta_2}\right)\, .$$
\end{lemma}
\begin{proof}
Due to \cite[Theorem 3]{CZ09},
\beq\label{bigk}
 \left|{\tilde\mu}(W_0W_k)\right|\le C \epsilon^2\theta_1^k\, .
\eeq
This estimate will be useful for big $k$ ($k\ge 2\log n/|\log \theta_1|$).
For $k=0$, due to Lemma \ref{Mm2}, there exists $c'>0$ such that
\beq\label{eqk=0}
{\tilde\mu}(W_0^2)
\le
\, \tilde\mu\left(V_0^2\right)
=\int_0^{\epsilon^{ 2}}\tilde\mu\left((Z_0/n^{\frac 1\alpha})^2>t\right)\, dt
\le c'\int_0^{\epsilon^2}t^{-\frac\alpha 2}n^{-1}\, dt= c'\frac{\epsilon^{2-\alpha }n^{-1}}{1-\frac \alpha 2}\, .
\eeq
Fix $\gamma\in(0,\frac 2\alpha-1)$.
Noticing that $$|W_k|\le \epsilon n^{\gamma-\frac 1\alpha} +\max\left(|W_k|-\epsilon n^{\gamma-\frac 1\alpha},0\right)$$ 
we obtain that
\begin{align}
 \left|{\tilde\mu}(W_0W_k)\right|
&\le \tilde\mu\left(\left|( \epsilon n^{\gamma-\frac 1\alpha}+|W_0|- \epsilon n^{\gamma-\frac 1\alpha})W_k\right|\right)\nonumber\\
&\le  \epsilon n^{\gamma-\frac 1\alpha} \tilde \mu\left(|W_k|\right)+ \tilde\mu\left(\max\left(|W_0|-\epsilon n^{\gamma-\frac 1\alpha},0\right)\(\max\left(|W_k|-\epsilon n^{\gamma-\frac 1\alpha},0\right)+\epsilon n^{\gamma-\frac 1\alpha}\)\right)\nn\\ 
&\le  2 \epsilon n^{\gamma-\frac 1\alpha} \tilde \mu\left(|W_0|\right)+ \tilde\mu\left(\max\left(|W_0|-\epsilon n^{\gamma-\frac 1\alpha},0\right)\max\left(|W_k|-\epsilon n^{\gamma-\frac 1\alpha},0\right)\right)\, .\label{smallk_1}
\end{align}
Due to the definition of $W_j$, $V_j$ and $Z_j$, we have
\beq\label{smallk_2}
\epsilon n^{\gamma-\frac 1\alpha} \tilde \mu\left(|W_0|\right)
\le 2\epsilon n^{\gamma-\frac 1\alpha} \tilde \mu\left(|V_0|\right)\le  4 \epsilon n^{-\frac 2\alpha+\gamma} \,K_0\, \tilde\mu\left(\mathcal R\right)\, ,
\eeq
where $K_0:= \max_{i\in \cJ}|A_i|$. 
Moreover,
\begin{eqnarray*}
&\ &
\tilde\mu\left(\max\left(|W_0|-\epsilon n^{\gamma-\frac 1\alpha},0\right)\max\left(|W_k|-\epsilon n^{\gamma-\frac 1\alpha},0\right)\right)\nonumber\\
&= &
\int_{(\epsilon n^{\gamma-\frac 1\alpha},3\epsilon)^2} \tilde\mu\left(|W_0|>r,\ |W_k|>s\right)\, dr\, ds\nonumber\\
&\le &
\int_{(\epsilon n^{\gamma-\frac 1\alpha},3\epsilon)^2} \tilde\mu\left(|V_0|>r-\frac{2\, K_0\tilde\mu(\mathcal R)}{n^{\frac 1\alpha}},\ |V_k|>s-\frac{2\, K_0\tilde\mu(\mathcal R)}{n^{\frac 1\alpha}}\right)\, dr\, ds\, ,
\end{eqnarray*}
since $\left|\tilde\mu(V_0)\right|\le \frac{2\, K_0\tilde\mu(\mathcal R)}{n^{\frac 1\alpha}}$.
Assuming $\epsilon$ and $n$ are such that
$\frac{2\, K_0\tilde\mu(\mathcal R)}{n^{\frac 1\alpha}}<\frac{\epsilon n^{\gamma-\frac 1\alpha}}2$, we get
\begin{eqnarray*}
&\ &
\tilde\mu\left(\max\left(|W_0|-\epsilon n^{\gamma-\frac 1\alpha},0\right)\max\left(|W_k|-\epsilon n^{\gamma-\frac 1\alpha},0\right)\right)\nonumber\\
&\le&\int_{(\epsilon n^{\gamma-\frac 1\alpha},3\epsilon)^2} \tilde\mu\left(|V_0|>r/2,\ |V_k|>s/2\right)\, dr\, ds\\
&\le &
\int_{(\epsilon n^{\gamma-\frac 1\alpha},3\epsilon)^2}\tilde\mu\left(
\mathcal R+\tilde\mu(\mathcal R)
>n^{\frac 1\alpha}r/(2K_0),\ \mathcal R\circ F^k+\tilde\mu(\mathcal R)>n^{\frac 1\alpha}s/(2K_0)\right)\, dr\, ds\, .\nonumber
\\
&\le &
\int_{(\epsilon n^{\gamma-\frac 1\alpha},3\epsilon)^2}\tilde\mu\left(
\mathcal R
>n^{\frac 1\alpha}r/(4K_0),\ \mathcal R\circ F^k>n^{\frac 1\alpha}s/(4K_0)\right)\, dr\, ds\, .\nonumber
\end{eqnarray*}
Now, using the inequality
$$ \tilde\mu\left({\mathcal R}>a,\ {\mathcal R}\circ F^k>b\right)\le\min\left(\tilde\mu\left({\mathcal R}>\min(a,b),\ {\mathcal R}\circ F^k>\min(a,b)\right),\tilde\mu\left({\mathcal R}>\max(a,b)\right)\right)\, ,$$
we obtain the existence of
$K'_0>0$ such that, for $\epsilon$ and $n$ as above and for every $k\in[1, (\epsilon n^\gamma)^\alpha]$,
\begin{eqnarray*}
&\ &
\tilde\mu\left(\max\left(|W_0|-\epsilon n^{\gamma-\frac 1\alpha},0\right)\max\left(|W_k|-\epsilon n^{\gamma-\frac 1\alpha},0\right)\right)\nonumber\\
&\le &
\int_{(\epsilon n^{\gamma-\frac 1\alpha},\epsilon )^2} K'_0\min\left(\left(\min(r,s)\right)^{-\alpha(1+\theta_0)}n^{-(1+\theta_0)}, (n\max(r,s)^\alpha)^{-1}\right)\, dr\, ds\, ,\label{smallk_3_step1}
\end{eqnarray*}
using \eqref{eq:condition D_r(u_n)2} (applied with $n$ of \eqref{eq:condition D_r(u_n)2} equal to $\lfloor n(\min(r,s))^\alpha\rfloor\ge \lfloor (\epsilon n^\gamma)^\alpha\rfloor\ge k$ and $\mathcal I=I_{c,c'}$ with $c=c'=1/(2K_0)$) for the first term in the last line and Lemma \ref{Mm2} for the second term in the last line. 
Therefore
\begin{eqnarray*}
&\ &\tilde\mu\left(\max\left(|W_0|-\epsilon n^{\gamma-\frac 1\alpha},0\right)\max\left(|W_k|-\epsilon n^{\gamma-\frac 1\alpha},0\right)\right)\\
&\le& 2K'_0\int_{\epsilon n^{\gamma-\frac 1\alpha}<r<s<\epsilon } \min\left((nr^\alpha)^{-1-\theta_0}, (ns^\alpha)^{-1}\right)\, dr\, ds\, .
\end{eqnarray*}
We assume from now on without loss of generality (up to restricting the value of $\theta_0$ if necessary) that
$(\alpha-1)(1+\theta_0)<1$.\\
Observe that $(nr^\alpha)^{-1-\theta_0}<(ns^\alpha)^{-1}$ happens if and only if $ s<n^{\frac{\theta_0}\alpha}r^{1+\theta_0}$ and that
$(ns^\alpha)^{-1}<(nr^\alpha)^{-1-\theta_0}$ happens if and only if $ r<n^{-\frac{\theta_0}{\alpha(1+\theta_0)}}s^{\frac 1{1+\theta_0}}$, which leads to
\begin{eqnarray}
&\ &
\tilde\mu\left(\max\left(|W_0|-\epsilon n^{\gamma-\frac 1\alpha},0\right)\max\left(|W_k|-\epsilon n^{\gamma-\frac 1\alpha},0\right)\right)\nonumber
\\
&\le& 2K'_0\left(\int_0^\epsilon  (nr^\alpha)^{-1-\theta_0}n^{\frac {\theta_0}\alpha}r^{1+\theta_0}\, dr+\int_0^\epsilon (ns^\alpha)^{-1}n^{-\frac{\theta_0}{\alpha(1+\theta_0)}}s^{\frac 1{1+\theta_0}}\, ds\right)\nonumber\\
&\le& 2K'_0\left(n^{-1-\theta_0+\frac {\theta_0}\alpha}\int_0^\epsilon  r^{-(\alpha-1)(1+\theta_0)}\, dr+n^{-1-\frac{\theta_0}{\alpha(1+\theta_0)}}\int_0^\epsilon s^{\frac 1{1+\theta_0}-\alpha}\, ds\right)\nonumber\\
&=&
\cO\left(\epsilon^{1-(\alpha-1)(1+\theta_0)}
n^{-1-\theta_0+\frac {\theta_0}\alpha}+\epsilon^{\frac
{1-(\alpha-1)(1+\theta_0)}
{1+\theta_0}}n^{-1-\frac{\theta_0}{\alpha(1+\theta_0)}}\right)=\cO\left(\epsilon^{\theta'_0}n^{-1-\theta'_0}\right)\, ,\label{smallk_3}
\end{eqnarray}
with $\theta'_0>0$ . Fix $\epsilon>0$ and consider $n$ large enough so that $\frac{K_0\tilde\mu(\mathcal R)}{n^{\frac 1\alpha}}<\frac{\epsilon n^{\gamma-\frac 1\alpha}}2$ and 
$2\log n/|\log \theta_1|\le (\epsilon n^\gamma)^\alpha$.
Putting together estimate \eqref{bigk} (for the sum over $k\ge 2\log n/|\log \theta_1|$), combined with  
\eqref{eqk=0}, \eqref{smallk_1}, \eqref{smallk_2} and \eqref{smallk_3} (for the sum over $k<2\log n/|\log \theta_1|$), we obtain that there exists $\tilde K>0$ such that for every $\epsilon>0$, for every $n$ large enough,
\beq
\sum_{k\ge 0} \left|\cov_{\tilde\mu}(V_0,V_k)\right|\le\tilde K
 \left( \epsilon^2 n^{-2}+\epsilon^{2-\alpha}n^{-1}
+\log n \left({\epsilon n^{-\frac 2\alpha+\gamma}}+\epsilon^{\theta'_0}n^{-1-\theta'_0}\right)\right)\, ,
\eeq
which ends the proof of the lemma.
\end{proof}
Armed with the above preliminary result, we can now prove a lemma which
	combined with \cite[Eq. (10.12) and Thm 10.1]{billingsley1999convergence} will directly implies Condition II.
\begin{lemma}\label{cond2lemma}
There exists $\tilde C_0>0$ and $\tilde\theta,\tilde\theta_0>0$ such that, for every $\epsilon\in(0,1)$ and for every $n$ large enough  (more precisely, there exists $ n_\epsilon$ such that for all $n>n_\epsilon$), and for every $0<{m_1< m_2}\le n$, we have
$$\tilde\mu \left(U_{m_1}^2({U_{m_2}-U_{m_1}})^2\right)\le \tilde C_0 \left(\frac{m_2}n\right)^{1+\tilde\theta}\left(\epsilon^{4-2\alpha}+\epsilon^{\tilde\theta_0}n^{-\tilde\theta_0}\right)\, .$$
\end{lemma}
\begin{proof}
We have
\begin{eqnarray}\label{eq:425}
\nn\tilde\mu\left(U_{m_1}^2(U_{m_2}-U_{m_1})^2\right)&=&\sum_{k_1,k_2=0}^{m_1-1}\sum_{k_3,k_4=m_1}^{m_2-1}
\tilde\mu\left(\prod_{j=1}^4W_{k_j}\right)\\
&\le& 4\, \sum_{0 \le k_1\le k_2<m_1< k_3\le k_4\le m_2-1}
\left|\tilde\mu\left(\prod_{j=1}^4W_{k_j}\right)\right|
\, . \label{varUn}
\end{eqnarray}
But, due to 
\cite[Theorem 3]{CZ09},
for every $j_0\in\{1,2,3\}$,
\beq\label{eq:426}
\left|\tilde\mu\left(\prod_{j=1}^4W_{k_j}\right)\right|
\le \left|\tilde\mu\left(\prod_{j\le j_0}W_{k_j}
\right)
\tilde\mu\left(\prod_{j> j_0}W_{k_j}\right)\right|
+ C'\, \epsilon^4 \theta_1^{k_{j_0+1}-k_{j_0}}\, .
\eeq
Notice that the first part of the right-hand side of \eqref{eq:426}
vanishes unless $j_0=2$. We apply \eqref{eq:426} to the sum on the right side of \eqref{eq:425}, by splitting up this sum according to the largest of the three gaps between the four indices $k_1\le k_2<k_3\le k_4$. We apply \eqref{eq:426} with $j_0=1$ (respectively 2 or 3) when the largest gap occurs between $k_1$ and $k_2$ (respectively $k_2$ and $k_3$ or $k_3$ and $k_4$), 
providing this gap is larger than or equal to $k$.
We conclude that, for every $k\ge 1$,
\begin{align*}
&\tilde\mu\left(U_{m_1}^2(U_{m_2}-U_{m_1})^2\right)\\&\le 4  \left(\sum_{\mathcal E_k}
{\tilde\mu}\left(\lb\prod_{j=1}^4 W_{k_j}\rb\right)+\left(\sum_{0\le k_1\le k_2<m_2}\left|\tilde\mu\left(\prod_{j=1}^2W_{k_j}\right)\right|\right)^2+  C'\, \epsilon^4 \, \sum_{\ell\ge k}\ell^4\theta_1^{\ell}\right)\, ,\nonumber
\end{align*}
where $\mathcal E_k$ is the set of $(k_1,k_2,k_3,k_4)$ satisfying
$0\le k_1\le k_2<m_1\le  k_3\le k_4< m_2$ such that $$\max(k_2-k_1,k_3-k_2,k_4-k_3)\le k.$$ 
Here $\ell$ in the final sum represents the largest gap, i.e., it is $\max(k_2-k_1,k_3-k_2,k_4-k_3)$, and the coefficient $\ell^4$ in the final sum comes from varying the $k_i$ subject to the constraint $0\le k_1\le k_2<m_1\le  k_3\le k_4< m_2$.

Now, Lemma \ref{lem:VarUn} ensures that, for every $\epsilon>0$, there exists $n_\epsilon$ such that, for every 
$n\ge n_\epsilon$ and every $0\le m_1<m_2$, we have 
\beq
\tilde\mu\left(U_{m_1}^2(U_{m_2}-U_{m_1})^2\right)\le 4 \sum_{\mathcal E_k}
{\tilde\mu}\left(\lb\prod_{j=1}^4 W_{k_j}\rb\right)+C''(m_2)^2(\epsilon^{4-2\alpha}n^{-2}+\epsilon^{2\theta_0}n^{-2-2\theta_2}) 
+ 4 C'\, \epsilon^4 \, \sum_{\ell\ge k}\ell^4\theta_1^{\ell}\, .
\label{bigkbis}
\eeq
Fix $\gamma\in(0,\frac 2\alpha-1)$.
Due to \eqref{smallk_1}, \eqref{smallk_2} and \eqref{smallk_3},  
there exists
$K'_0>0$ and $\theta'_0>0$ such that, for every $\epsilon$ and $n$
satisfying $K_0\tilde\mu(\mathcal R)<\frac{\epsilon n^{\gamma}}2$, for every $(k_1,k_2,k_3,k_4)$  such that $\max_{j=1,2,3}(k_{j+1}-k_{j})\in[1, (\epsilon n^\gamma)^\alpha]$,
\beq
\tilde\mu\left(\left|\prod_{j=1}^4W_{k_j}\right|\right)\le
4\epsilon^2 \tilde\mu\left(\left|W_{k_1}W_{k_2}\right|\right)\le K'_0 
\epsilon^{\theta'_0}n^{-1-\theta'_0}\, .
\label{smallkbis}
\eeq
Fix $\epsilon>0$ and consider $n$ large enough so that
$K_0\tilde\mu(\mathcal R)<\frac{\epsilon n^{\gamma}}2$ and such that 
$k=3\log n/|\log \theta_1|\le (\epsilon n^\gamma)^\alpha$.
Putting together \eqref{bigkbis} with \eqref{smallkbis}, we obtain that
\begin{eqnarray*}
\tilde\mu\left(U_{m_1}^2(U_{m_2}-U_{m_1})^2\right)&\le& C''' \left(\epsilon^ 4\, n^{-2}+\frac{m_2^2}{n^2}(\epsilon^{4-2\alpha}+n^{-2\theta_2})+(\log n)^4 n^{-1-\theta'_0}\right)\\
&\le& C''' \left(\epsilon^ 4\, \left(\frac {m_2}n\right)^{2}+\left(\frac {m_2}n\right)^{2}(\epsilon^{4-2\alpha}+n^{-2\theta_2})+(\log n)^4 \left(\frac{m_2}{n}\right)^{1+\tilde\theta}n^{-\frac{\theta'_0}2}\right)
\end{eqnarray*}
with $\tilde\theta:=\min(1,\frac{\theta'_0}2)$ since
$n^{-1-\theta'_0}\le  \frac{m_2^{1+\tilde\theta}}{n^{1+\theta'_0}}\le \left(\frac{m_2}{n}\right)^{1+\tilde\theta}n^{-\frac{\theta'_0}2}$.
This ends the proof of Lemma \ref{cond2lemma}.
\end{proof}
\begin{proof}[Proof of Condition II]
Due to the Markov inequality and to Lemma \ref{cond2lemma}, for every $0\le i\le j\le k\le n$,
\begin{eqnarray*}
\tilde\mu\left(|U_j-U_i|\wedge|U_k-U_j|\ge\lambda\right)&\le&\lambda^{-4}
\mathbb E_{\tilde\mu}\left[|U_j-U_i|^4\wedge|U_k-U_j|^4\right]\\
&\le& \lambda^{-4}
\mathbb E_{\tilde\mu}\left[|U_j-U_i|^2|U_k-U_j|^2\right]\\
&\le& \lambda^{-4}
\mathbb E_{\tilde\mu}\left[|U_{j-i}|^2|U_{k-i}-U_{j-i}|^2\right]\\
&\le& \lambda^{-4}\tilde C_0 \left(\frac{k-i}n\right)^{1+\tilde\theta}\left(\epsilon^{4-2\alpha}+\epsilon^{\tilde\theta_0}n^{-\tilde\theta_0}\right)\, .
\end{eqnarray*}
Therefore, due to Theorem \cite[Theorem 10.1]{billingsley1999convergence} (with $\beta=1$ and $\alpha=(1+\tilde\theta)/2$ and $u_\ell=n^{-1}(\tilde C_0(\epsilon^{4-2\alpha}+n^{-\tilde\theta_0}))^{1/(1+\tilde\theta)}$), there exists $K>0$ depending only on $(\alpha,\beta)$ such that
\beq
\tilde\mu\left(L_n\ge \lambda\right)\\
\le \lambda^{-4}K\tilde C_0 \left(\epsilon^{4-2\alpha}+\epsilon^{\tilde\theta_0}n^{-\tilde\theta_0}\right)\, ,
\eeq
with $L_n:=\max_{0\le i\le j\le k\le n}(|U_j-U_i|\wedge|U_k-U_j|)$.
As noticed in \cite[(10.3)]{billingsley1999convergence},
$$\max_{k\le n}|U_k|\le  3L_n+\max_{k\le n}|V_k-\tilde\mu(V_0)|\, .$$
Therefore, for every $\gamma>0$,
\begin{eqnarray*}
\tilde\mu\left(\max_{0\le k\le n}|U_k|\ge \gamma\right)&\le&\tilde\mu\left(L_n\ge \frac{\gamma}6\right)+\tilde\mu\left(\max_{k\le n}|V_k-\tilde\mu(V_0)|\ge \frac\gamma 2\right)\\
&\le& K\, 6^4\gamma^{-4}\tilde C_0 \left(\epsilon^{4-2\alpha}+\epsilon^{\tilde\theta_0}n^{-\tilde\theta_0}\right)+  \tilde\mu\left(\max_{k\le n}|V_k|\ge \frac{\gamma}2-\tilde\mu(V_0)\right)\\
&\le& K\, 6^4\gamma^{-4}\tilde C_0 \left(\epsilon^{4-2\alpha}+\epsilon^{\tilde\theta_0}n^{-\tilde\theta_0}\right)+  \tilde\mu\left(\epsilon \ge \frac{\gamma}2-\frac{\tilde\mu(|Z_0|)}{n^{\frac 1\alpha}}\right)\, .
\end{eqnarray*}
Consequently, for every $\gamma>0$,
\beq
\lim_{\varepsilon\rightarrow 0}\limsup_{n\rightarrow +\infty}\tilde\mu\left(\max_{0\le k\le n}|U_k|>\gamma\right)
\le\lim_{\varepsilon\rightarrow 0}\left[K 6^4\gamma^{-4}\tilde C_0 \epsilon^{4-2\alpha}+  \tilde\mu\left(\epsilon \ge 
\frac{\gamma}2\right)\right]=0\, .
\eeq
\end{proof}

\begin{appendix}
\section{Asymmetric cusps with  flatness $\beta$}
In this section, we concentrate on one  cusp $P_i$ with flatness $\beta$ for $i\in \cJ$. 
To simplify notation, we ignore the index $i$, and consider here the cusp formed by $\Gamma, \Gamma'$ with a common tangent line at the end point $P$.  
{To simplify notation, we assume that the flat point $P$ has curvilinear abscissa $r=0$ and $r'=|\partial Q|$. }
We choose a Cartesian coordinate system $(s,z)$ with origin at $P$, and with the horizontal $s$-axis being the tangent line to the boundary of the billiard table. By \eqref{z1s},     for some small $\eps_0>0$,  the boundary pair $\Gamma$ and $\Gamma'$ adjacent to $P$ can be represented in the $\eps_0-$neighborhood of the cusp $P$ as:
$\{(s,z_0(s)),\, s\in[0,\eps_0]\}\cup \{(s,-z_1(s)),\, s\in[0,\eps_0]\}$, with
 \beq\label{zsP}
 z_i(s)= c_i\beta^{-1} s^{\beta}+\cO(s^{2\beta-1}),\,\,\,\,\,\, 
 z'_i(s)= c_i s^{\beta-1}+\cO(s^{2\beta-2}),\,\,\,\,\,\,\forall s\in [0,\eps_0]\, ,\eeq
where  $c_0, c_1\ge 0$ not both equal to 0.
We also set
$
\bar c=\frac{c_0+c_1}2
$
and $\alpha:=\beta/(\beta-1)$.\\
We write $\tilde \cM$ for the set of vectors in $\cM$ that are in the cusp area $B_\epsilon(P)$ and such that the previous reflection off of $\partial Q$ was outside the cusp area $B_\epsilon(P)$. Fix $N_0$.
For any $N\geq N_0$, we define $M_N$ to be  the set of points in $\tilde \cM$ whose forward trajectories explore the cusp at $P$ over $N$ reflections off of the first curve, either $\Gamma$ or $\Gamma'$, that it hits, before leaving the cusp.
\subsection{The corner series}
In this subsection,  we  investigate the geometry of corner series, which correspond to  certain billiard trajectories
entering an asymmetric cusp of flatness $\beta>0$ and  experience a large number of refections there before
exiting.  
For a large $N\ge N_0$, we consider a corner sequence entering the cusp at $P$ with $N$ reflection off the first curve it hits before leaving the cusp  (so making either $2N$ or $2N+1$ reflections in the cusp).

We assume moreover (up to permuting the roles played by $\Gamma$ and by $\Gamma'$) that the first reflection is on $\Gamma$.
We denote $(x_n)_n=((r_n,\varphi_n))_n$ as the consecutive sequence colliding with the boundary $\Gamma$, and $(x_n')_n=((r_n',\varphi_n'))_n$ the sequence on $\Gamma'$. Let $s_n$ (resp. $s_n'$) to be the $s$-coordinate of the base point of $x_n$ (resp. $x_n'$), for $n=1,\cdots, N$. By the smoothness of the boundary curves,
\beq\label{rnsn}|r_n|=\int_0^{s_n}\sqrt{1+(z'_0(u))^2}\, du=s_n+\cO(s_n^{2\beta-1}),\,\,\,\,\,\, |r_n'|=s_n'+\cO((s_n')^{2\beta-1})\, . \eeq
To estimate the tail distribution of $\mu_M({\cal R}\geq n)$, we will fix $N_0$ (as above), and only consider those corner series, such that $N\ge N_0$. We will also work
with more convenient coordinates:
$$\gamma_n=\min(\varphi_n,\pi-\varphi_n)
, \,\,\,\,\,\,\,
\gamma_n'=\min(\varphi'_n,\pi-\varphi'_n)
\,\,\,\,\,\,\text{ and }\,\,\,\,\,\,\,\alpha_n=\arctan(z'_0(s_n)),\,\,\,\,\,\,\alpha_n'=\arctan(z'_1(s_n'))\, .$$
Note that by  (\ref{zsP}),  the tangent vector of $\partial Q$ at
$(s_n,z_i(s_n))$ is $(1, z_i'(s_n))$,
which implies, by Taylor-expanding, that
\beq \label{alphansn}\alpha_n=\arctan(z'_0(s_n))=c_0s_n^{\beta-1}+\cO(s_n^{2\beta-2}),\,\,\,\,\alpha_n'=\arctan(z'_1(s'_n))=c_1(s_n')^{\beta-1}+\cO((s_n')^{2\beta-2})\, ,\eeq
where $\alpha_n$ (resp. $\alpha'_n$) stands for the angle in $[0,\frac \pi 2]$ of the tangent line to $\Gamma$ at $(s_n,z_0(s_n))$ (resp. $\Gamma'$ at $(s'_n,-z_1(s'_n))$)
made with the horizontal axis, or equivalently, with the tangent line through the flat point $P$.  Note  that both $\alpha_n$ and
$\gamma_n$ are positive for $1\leq n\leq N-1$.
The sequences $(\alpha_n)_n$ and $(\alpha'_n)_n$ are decreasing and take small values if $N$ is  large enough. While the $\gamma_n$
are initially small, they slowly grow to
about $\pi/2$ for $n\sim N/2$, and then again decrease and
get small.  We use notation similar to that of \cite{chernov2007dispersing}, and define
$N_2$ such that
$$\alpha_{N_2}:=\min\{\alpha_n\,:\, 1\leq n\leq N\}\, .$$
Comparing the trajectory
 $(T^{N_2-2j}x)_{j=0,...,\lfloor N_2/2\rfloor}$
 with the outgoing trajectory $(T^{N_2+2j}x)_{j=0,...,\lfloor (N-N_2)/2\rfloor}$, we conclude that only two cases can occur
$$\mbox{either }s_{N_2-1}<s_{N_2+1}<s_{N_2-2}<s_{N_2+2}<...\mbox{ or  }s_{N_2+1}<s_{N_2-1}<s_{N_2+2}<s_{N_2-2}<... \, .$$
This implies that $|N_2-N/2|\leq 2$.
We further subdivide the corner series into three segments. We fix a small enough
$\bar\gamma$ and  set $$N_1:=\max\{ n\leq N_2\,:\, \gamma_n<\bar\gamma\},\,\,\,\, N_3:=\max\{n\geq N_2\,:\, \gamma_n>\bar\gamma\}.$$
We call the segment on $[1, N_1]$ the ``entering period'' in the corner series, the
segment $[N_1 + 1;N_3-1]$ the ``turning period'', and the segment $[N_3,N]$
the ``exiting period''. The same argument as above for $N_2$ shows that $|(N_3-N_2)-(N_2-N_1)|\le 2$ so that
{$$|N_3+N_1-N|\le |(N_3-N_2)-(N_2-N_1)|+|2N_2-N|$$
implies
$|N_3+N_1-N|\leq 6$.}
By the reversibility of the billiard dynamics, it is enough to consider the first half of the series, $1\leq n\leq N_2$.

Using the relations above, we collect various estimates in the following proposition for a corner series of length $N$ generated by any $x\in M_N$.
We first denote an important function $H$ defined on $\mathcal M$ in the cusp by
$$H(r,\varphi):=|r|^\beta\sin\varphi\, . $$
Recall that here $|r|$ represents the curvilinear distance on $\partial Q$ between the cusp $P$ and the point $x=(r,\varphi)$ in the cusp that we are interested in.
For every $n=1,...,N$, we set
$$H_n:=H(r_n,\varphi_n) =|r_n|^\beta\sin\varphi_n\quad\mbox{and}\quad H'_n:=H(r'_n,\varphi'_n) =|r'_n|^\beta\sin\varphi'_n\, .$$
\begin{proposition}\label{prop3} 
Assume $\beta_0=\beta_1$ or $\max(\beta_0,\beta_1)\ge 2\beta-1$.
The following are true:
\footnote{where $a_n\approx b_n$ means that, for $n$ large enough, there exist $\tilde c,\tilde C>0$ (independent of the corner series  and of $n$) such that $\tilde c\, a_n\le b_n\le \tilde C\, a_n$.}\\
(1)  $N_1\approx N_2-N_1\approx N_3-N_2\approx N-N_3\approx N,$ i.e. all  three segments in the corner series have length of order $N$;\\
(2) $s_2\approx s'_1\approx N^{-\frac{\beta}{(2\beta-1)(\beta-1)}},\,\,\,\, s_{n}\approx s'_n\approx n^{-\frac{1}{\beta-1}}\sim N^{-\frac{1}{\beta-1}}$, for $n\in [N_1, N_2]$;\\
(3) $s_n\approx s'_n\approx (nN^{\frac{\beta}{\beta-1}})^{-\frac{1}{2\beta-1}}$, for
$n\in [2,N_2]$;\\
(4) $\gamma_1\approx \gamma_1'=\cO(N^{-\frac{\beta}{2\beta-1}}), \,\,\gamma_2\approx \gamma_2'\approx N^{-\frac{\beta}{2\beta-1}}$;\\
(5) $
v_n\approx v'_n\approx
\gamma_n\approx \gamma_n'\approx(n N^{-1})^{\frac{\beta}{2\beta-1}}$, for 
$n\in [2,N_2]$;\\
(6) For  $N$ sufficiently large, the  quantity $\{H_{N}((r_n,\varphi_n)), n=1,\cdots, N-1\}$ and $\{H_{N}'((r_n',\varphi_n')), n=1,\cdots, N-1\}$ are both almost invariant along a corner series of length $N$:
$$\forall n=1,...,N_2,\quad H_{n}=C_N+\cO(s_n^{2\beta-1})
\quad\mbox{and}\quad H_{n}'=C'_N+\cO(s_n^{2\beta-1})
\, .$$
\end{proposition}
More precisely, we will see in Section \ref{sec:scale} that $C_N=\bar c^{-\alpha} I_1^\alpha N^{-\frac\beta{\beta-1}}+\cO\left(N^{-\frac{2\beta-1}{\beta-1}}\ln N\right)$ and that $C'_N=\bar c^{-\alpha} I_1^\alpha N^{-\frac\beta{\beta-1}}+\cO\left(N^{-\frac{2\beta-1}{\beta-1}}\ln N\right)$ uniformly in $x\in M_N$.
\begin{proof}
For $n=1,\cdots, N_2$,
both $\{\alpha_{n}\}$ and $\{\alpha_n'\}$ are decreasing sequences, and $\{\gamma_{n}\}$ and $\{\gamma_n'\}$ are  increasing sequences. 
Observe that
\beq\label{3.6b}
\gamma_{n+1}=\gamma'_n+\alpha'_n+\alpha_{n+1}\quad\mbox{and}\quad
\gamma'_{n}=\gamma_n+\alpha_n+\alpha'_{n}\, .
\eeq
Now we denote $v_n:=\gamma_n+\alpha_n$ and $v_n':=\gamma_n'+\alpha_n'$. Observe that $v_n$ and $v'_n$ correspond to the angles between the horizontal line (i.e. the tangent line to the cusp) and the reflected directions corresponding to $(r_n,\varphi_n)$ and $(r'_n,\varphi'_n)$, respectively. With this notation, \eqref{3.6b} leads to
\beq\label{3.6}
v_{n+1}=v_{n}'+2 \alpha_{n+1},\,\,\,\,\,v_{n}'=v_{n}+2 \alpha_{n}'\, .\eeq
This also implies that
\beq\label{3.612}
v_{n}=v_{n-1}+2 \alpha_{n}+2 \alpha_{n-1}',\,\,\,\,\,\,\,\,\,v_{n}'=v_{n-1}'+2 \alpha_{n}+2 \alpha_{n}'\, .\eeq
By (\ref{3.6}) and (\ref{3.612}), we have
\beq\label{3.8}
v_2>2\alpha_2,\,\,\,\,\sum_{n=1}^{N_2}\alpha_n\leq v_{N_2}/2\leq \pi/4
\eeq
and
\beq\label{3.8bis}
(n-1)2\bar c\, (s_n^{\beta-1}+\cO(s_n^{2\beta-2}))\le (n-1)(\alpha_n+\alpha'_{n-1})\le\sum_{k=2}^{n}(\alpha_n+\alpha'_{n-1})\leq v_n/2\leq 
   \min\left(\frac\pi 4\sin v_n,\frac{\tan v_n}2\right)\, .
\eeq
This implies in particular that
\beq\label{asymptots_n}
s_n=\cO\left(n^{-\frac 1{\beta-1}}\right)\, .
\eeq
We denote
\beq\label{tau2n-1}
\tau_{n}:= \frac{z_0(s_{n})+z_1(s_{n}')}{\sin v_{n}},\,\,\,\,\,\,\tau_{n}':= \frac{z_0(s_{n+1})+
z_1(s'_{n})
}{\sin v_{n}'}.\eeq
Here $\tau_{n}$ is the free path between two collisions based at $x_{n}$ and $x_n'$, while $\tau_{n}'$ is the free path between two collisions based at $x_{n}'$ and $x_{n+1}$. Observe that
\beq\label{3.7}
s_{n}=s_{n-1}'-\tau'_{n-1}
\cos v'_{n-1}
=s_{n-1}'
-  \frac{z_0(s_{n})+z_1(s_{n-1}')}{ \tan v'_{n-1}},\quad s_{n}'=s_{n}-\tau_n
\cos v_n
=s_{n}-  \frac{z_0(s_{n})+z_1(s_{n}')}{ \tan v_{n}}\, .
\eeq
This implies that
\beq
\label{s2n+1b}
s_{n+1}-s_{n}=\cO\left(\frac{s_n^\beta}{\tan v_n}\right)
\eeq
which combined with \eqref{3.8bis} implies that $s_{n+1}/s_n=1+\cO(s_n^{\beta-1}/\tan v_n)=1+\cO(1/n)$ and that 
\beq \label{snsn+1}
s_n/s_{n+1}=\cO(1)\, .
\eeq 
More precisely, we obtain
that
\begin{align}
s_{n+1}-s_{n}&=-  \frac{z_0(s_{n+1})+z_0(s_{n})}{ \tan v_{n}}-  \frac{2z_1(s_{n}')}{ \tan v_{n}}+\cO\left(\frac{s_{n}^{2\beta-1}}{(\sin v_{n})^2}\right)\label{diffsn}\\
&=-\frac {4\bar c s_n^\beta}{\beta\tan v_n}+\cO\left(\frac{s_{n}^{2\beta-1}}{(\sin v_{n})^2}\right)\, .\label{s2n+1}
\end{align}
where we used the fact that $\frac 1{\tan v_n}-\frac 1{\tan v'_n}=\cO\left(\frac{|v_n-v'_n|}{\sin^2 v_n}\right)$ combined with
\eqref{3.6} and \eqref{zsP}.
We denote
$$A_n:= s_n^{\beta}\sin v_n,\quad n=1,\cdots, N_2\, .$$
The next lemma tells that $A_n\sim A_{N_2}$ is almost invariant,  as long as $N$ is large. 
We set 
$\tilde N_2:=\max\{n<N_2\, :\, v_n<\frac \pi 2-\bar\eta_1\}$
(for any fixed $\bar\eta_1>0$)
and prove the following estimates.
\begin{lemma}
	\beq\label{diffAn}
A_{n+1}-A_n= \cO\left(\frac{s_{n}^{3\beta-2}}{\sin v_n}\right)\, .\eeq
Moreover
\beq\label{Hn1}
\forall n=1,\cdots,  \tilde N_2,\quad A_{n} = A_{\tilde N_2}+\cO(s_n^{2\beta-1})\, ,
\eeq
and 
for every $\boldsymbol{\varepsilon}\in(0,1)$,
\beq\label{Hn2}
\forall n=1,\cdots,   N_2,\quad A_n=A_{\tilde N_2}+O(s_n^{2\beta-1-\boldsymbol{\varepsilon}})\, .
\eeq
\end{lemma}
\begin{proof}
We start by writing
$$A_{n+1}-A_n=(s_{n+1}^{\beta}-s_{n}^{\beta})(\sin v_n+\sin v_{n+1}-\sin v_n)+s_n^{\beta}(\sin v_{n+1}-\sin v_n)\, .$$
Taylor expansions at the first order combined with \eqref{s2n+1} and with \eqref{3.612} lead to
$$s_{n+1}^\beta-s_n^\beta=\beta s_n^{\beta-1}(s_{n+1}-s_n)+\cO\left(s_n^{\beta-2}(s_{n+1}-s_n)^2\right)
=-\frac{4\bar cs_n^{2\beta-1}}{\tan v_n}+\cO\left(\frac {s_n^{3\beta-2}}{(\sin v_n)^2}\right)\, ,$$
and
\begin{align*}
\sin v_{n+1}-\sin v_n&=\cos v_n\left(v_{n+1}-v_n\right)+\cO\left(v_n\left(v_{n+1}-v_n\right)^2\right)\\
&=2\cos v_n\left(\alpha_{n+1}+\alpha'_n\right)+\cO\left( s_n^{2\beta-2}\right)\\
&=4\bar c s_n^{\beta-1}\cos v_n+\cO\left(s_n^{2\beta-2}/\tan v_n\right)\, . 
\end{align*}
Therefore
\beq
A_{n+1}-A_{n}=
\cO\left(\frac{s_{n}^{3\beta-2}}{\sin v_n}\right)
\label{en}\, .
\eeq
	Using \eqref{s2n+1} and \eqref{snsn+1}, it comes that
	\begin{align}
	s_{n+1}^{2\beta-1}-s_n^{2\beta-1}&=-\frac{4\bar c(2\beta-1)}{\beta}\frac{s_{n}^{3\beta-2}}{ \tan v_n}+\cO\left(\frac{s_n^{4\beta-3}}{\sin^2 v_n}\right)\nonumber\\
	&=-\frac{4\bar c(2\beta-1)}{\beta}\frac{s_{n}^{3\beta-2}}{ \tan v_n}+\cO\left(\frac{s_n^{2\beta-1}}{n^2}\right)\, ,\label{sn2b-1}
	\end{align}
	due to \eqref{3.8bis}, which implies that
$$\forall n=1,...,N_2,\quad \sum_{k=n}^{N_2} \frac{s_k^{3\beta-2}}{\tan v_k}=\cO\left( s_n^{2\beta-1}\right)\, ,$$
and so that
$$
\forall n=1,...,\tilde N_2,\quad \left|A_n-A_{\tilde N_2}\right|\le \sum_{k=n}^{\tilde N_2-1}\left|A_{k+1}-A_k\right|=\cO\left( s_n^{2\beta-1}\right)\, .$$
Let $\mathbf{u}\in(0,1)$ be such that $\beta+(\beta-1)(1-\mathbf{u})= 2\beta-1-\boldsymbol{\varepsilon}$.
For every $n= 1,...,N_2$, 
\begin{align*}
A_n-A_{N_2}&=\cO\left(\sum_{k=n}^{N_2} \frac{s_k^{3\beta-2}}{\sin v_k}\right)=\cO\left(s_n^{\beta+(\beta-1)(1-\mathbf{u})}\sum_{k=n}^{N_2} \frac{s_k^{(\beta-1)(1+\mathbf{u})}}{\sin v_k}\right)\\
&=\cO\left(s_n^{\beta+(\beta-1)(1-\mathbf{u})}\sum_{k=n}^{N_2} \frac{(\sin v_k/k)^{1+\mathbf{u}}}{\sin v_k}\right)=\cO(s_n^{\beta+(\beta-1)(1-\mathbf{u})})\, , 
\end{align*}
using \eqref{3.8bis}.
\end{proof}
 We define \beq\label{defCN}
C_N:=A_{\tilde N_2}=s_{\tilde N_2}^{\beta}\sin v_{\tilde N_2}\,.\eeq
Now let us prove the key estimates 
\begin{itemize}
	\item For every $n=N_1,...,N_2$, $\sin v_n\approx 1$ and so $s_n^\beta\approx s_{N_1}^\beta\approx s_{N_2}^\beta\approx C_N$. Therefore
	\beq\label{AAA1}\forall n=N_1,...,N_2,\quad s_n\approx C_N^{\frac 1\beta}.\eeq
	\item $1\approx v_{N_2}-v_{N_1}\approx
	   \sum_{k=N_1}^{N_2}s_k^{\beta-1}\approx(N_2-N_1)C_N^{\frac {\beta-1}\beta}$. Therefore
	   \beq\label{AAA2}
	    N_2-N_1\approx C_N^{-\frac{\beta-1}\beta}\, . \eeq
\item Let us see that $v_1=O(s_1^{\beta-1})$. This is obviously true if $s_1\ge \epsilon_0/2$.
Assume now that $s_1< \epsilon_0/2$, then by definition of $s_1$ the incident line
at $(s_1,z(z_1))$ intersects $[(2s_1,z_0(2s_1)),(2s_1,-z_1(2s_1))]$, which implies that
$\tan (v_1-\alpha_1)\le \cO(s_1^\beta)/s_1=\cO(s_1^{\beta-1}) $, and so $v_1=\cO(s_1^{\beta-1})$.
\\
 This combined with \eqref{3.612} and \eqref{snsn+1} implies that 
	$v_2\approx s_2^{\beta-1}$ and so
	$C_N\approx s_2^\beta v_2\approx  s_2^{2\beta-1}$, and so
	\beq\label{s2}
	 s_2\approx{C_N}^{\frac 1{2\beta-1}}\, . \eeq
	 \item For every $n=2,...,N_1$, $s_{n+1}-s_n\approx\frac{s_n^\beta}{v_n}\approx\frac{s_n^{\beta}}{C_Ns_n^{-\beta}}\approx \frac{s_n^{2\beta}}{C_N}$ and so
	 $$s_{n+1}^{-2\beta+1}-s_{n}^{-2\beta+1}\approx s_n^{-2\beta}(s_{n+1}-s_n)\approx C_N^{-1} \, .$$
	 Moreover, due to \eqref{s2}, since $s_2^{-2\beta+1}\approx C_N^{-1}$. Therefore
	 \beq\label{AAA3}
	 \forall n=2,...,N_1,\quad  s_n\approx\left( s_n^{-2\beta+1}\right)^{-\frac 1{2\beta-1}} 
	 \approx\left(\frac {n}{C_N}\right)^{-\frac 1{2\beta-1}}\approx\left(\frac {C_N}{n}\right)^{\frac 1{2\beta-1}}\, .\eeq
	 \item Due to \eqref{AAA3}, $s_{N_1}\approx (C_N/N_1)^{\frac 1{2\beta-1}}$. This and \eqref{AAA1} leads to $(C_N/N_1)^{\frac 1{2\beta-1}}\approx C_N^{\frac 1\beta}$. Therefore
	 \beq\label{AAA4} N_1\approx C_N^{-\frac{\beta-1}\beta}\, .\eeq
	 Combining this with \eqref{AAA2}, we obtain
	 \beq 
	 N_1\approx N_2-N_1\approx N\approx C_N^{-\frac{\beta-1}{\beta}}
	 \eeq
	 and so in particular
	 \beq \label{CN}
	 C_N\approx N^{-\frac{\beta}{\beta-1}} \, .
	 \eeq
	 \item Combining \eqref{CN} with respectively \eqref{AAA3} and \eqref{AAA1},we obtain
	 \beq\label{AAA3b}
	 \forall n=2,...,N_1,\quad  s_n\approx\left(n\,  N^{\frac{\beta}{\beta-1}}\right)^{-\frac 1{2\beta-1}}\, .\eeq
	 \beq\label{AAA1b}
	 \forall n=N_1,...,N_2,\quad  s_n\approx N^{-\frac 1{\beta-1}}\approx\left(n\,  N^{\frac{\beta}{\beta-1}}\right)^{-\frac 1{2\beta-1}}\, ,\eeq
	 since $N_1\approx N_2\approx N$.
\end{itemize}
This further implies that for $n\in [2,N_2]$, we have
\beq\label{N1N2}
s_n^{\beta-1}\approx (\alpha_n+\alpha'_{n-1})\approx (n^{\beta-1}N^{\beta})^{-\frac{1}{2\beta-1}},\,\,\,\,\,\gamma_n\approx v_n\approx (n N^{-1})^{\frac{\beta}{2\beta-1}}.\eeq
In particularly, for $n\in [N_1,N_2]$, using the above fact that $N_1\approx N_2\approx N$, we have
\beq\label{N1N22}(\alpha_n+\alpha'_{n-1})\approx s_n^{\beta-1}\approx n^{-1}\approx N^{-1},\,\,\,\,\,v_n\approx 1.
\eeq
Now by \eqref{Hn1}, we know that
$$\forall n=1,...,\tilde N_2,\quad A_n=s_n^{\beta}\sin v_n=C_N+\cO(s_n^{2\beta-1})\, .$$
Due to \eqref{diffAn} and to \eqref{N1N22},
$$\forall n=\tilde N_2,...,N_2,\quad  
      A_{N_2}-A_n=\cO\left(\sum_{k=\tilde N_2}^{N_2} s_n^{3\beta-2}\right) =
      \cO\left( s_n^{1-\beta} s_n^{3\beta-2} \right)=\cO\left(s_n^{2\beta-1}\right)\, .$$ 
Due to \eqref{rnsn},
we have
\beq\label{defnHn}
H_{n}=|r_n|^{\beta}\sin \varphi_n= s_n^{\beta}\sin\gamma_n+\cO(s_n^{3\beta-2}\sin\gamma_n).
\eeq
The above estimation implies that
\begin{align*}
H_{n}&=A_n-s_n^{\beta}(\sin \gamma_n-\sin v_n)+\cO(s_n^{3\beta-2}\sin\gamma_n)\\
&=A_n+2s_n^{\beta} \sin\frac{\alpha_n}{2} \, \cos(\gamma_n+\frac{\alpha_n}{2})+\cO(s_n^{3\beta-2}\sin\gamma_n)\\
&=A_n+\cO(s_n^{2\beta-1})
=C_N+\cO(s_n^{2\beta-1})\, ,
\end{align*}
where we used $\alpha_n=\cO(s_n^{\beta-1})$.
We also observe that
\beq\label{defnH'n}
H'_{n}=|r'_n|^{\beta}\sin \varphi_n'= (s_n')^{\beta}\sin(\gamma_n')+\cO(s_n^{2\beta}\sin\gamma_n).
\eeq
The above estimation implies that
$H'_{n}=C_N+\cO(s_n^{2\beta-1})$.
\end{proof}


 Due to the time reversibility of billiard dynamics, all the asymptotic formulas obtained for the entering period remain valid for the exiting period.

%
%
%
%
%
\subsection{Proof of Lemma \ref{Mm2}}\label{sec:scale}
Recall that $\alpha=\frac{\beta}{\beta-1}$. 
We set $\nu:=2|\partial Q|\mu$ for the non normalized measure on $\mathcal M$ simply given by $d\nu=\sin\varphi\, dr\, d\varphi$. Write $\tilde M$ for the set of vectors in $\mathcal M$ that are in the cusp area $B_\epsilon(P)$ and such that the previous reflection off $Q$ was outside the cusp area $B_\epsilon(P)$. 
The purpose of this section is to prove that

\beq\label{eq:scale parameter2BB}
\lim_{y\to\infty} y^\alpha \nu\left(\tilde M\, :\, \mathcal R > y\right)=
  \frac 4\beta\bar c^{-\frac 1{\beta-1}}I_1^{\alpha} \,,
\eeq
which will imply Lemma \ref{Mm2} since $\nu$ is $T$-invariant and since 
$\tilde\mu=\frac{\nu_{|M}}
   {\mu(M)|\partial Q|}$.
Recall that we denote 
$ T ^n x=(r_n,\varphi_n)$
and $\gamma_n=\min(\varphi_n, \pi-\varphi_n)$. 
Due to Proposition \ref{prop3}, for $N$ large enough  the following sequences are almost constant for $n=1,\ldots, N$:
\beq\label{hnB} 
H_n=C_N+\cO(s_n^{2\beta-1}), H_n'=C_N'+\cO((s_n')^{2\beta-1})\, ,
\eeq
where $C_N=\cO\left( N^{-\alpha}\right) $, $C_N'=\cO\left( N^{-\alpha}\right) $.
In order to estimate $C_N$ and $C_N'$, we use an elliptic integral and introduce
\beq\nn w_n:=\int_0^{\gamma_n} (\sin u)^{1-\frac{1}{\beta}}\, du,\eeq
 for $n=1,\ldots, N_2$.
Then
\beq\label{vn1B} w_{n+1}-w_n=\int_{\gamma_{n}}^{\gamma_{n+1}} (\sin u)^{1-\frac{1}{\beta}}\, du=(\sin\gamma_n^*)^{1-\frac{1}{\beta}} (\gamma_{n+1}-\gamma_{n})\eeq
for some $\gamma_n^*\in [\gamma_{n},\gamma_{n+1}]$. 
Due to \eqref{defnHn} and \eqref{hnB},
we have
\beq\label{hn2B}
\sin\gamma_n=\frac{H_n}{r_n^{\beta} }=\frac{H_n}{s_n^\beta}+\cO\left(s_n^{2\beta-2}\right)=\frac{C_N}{s_n^\beta}+\cO(s_n^{\beta-1}).
\eeq
By \eqref{3.6b} and \eqref{s2n+1b},
\begin{align}\label{phin+1nB}
\gamma_{n+1}-\gamma_{n}=2\alpha'_n+\alpha_n+\alpha_{n+1}=4\bar c s_n^{\beta-1}+\cO\(s_n^{2\beta-2}/\tan v_n\).
\end{align}
and $\gamma_{n+1}-\gamma_n=\cO(s_n^{\beta-1})$.
Now combining the above and recalling that $\alpha=\beta/(\beta-1)$, we rewrite \eqref{vn1B} as
\begin{align}
&w_{n+1}-w_n=(\sin^{\frac {\beta-1}{\beta}}\gamma_n)\, (\gamma_{n+1}-\gamma_n)+\cO\left(\sin^{-\frac 1\beta}\gamma_n\, (\gamma_{n+1}-\gamma_n)^2\right)\nonumber\\
&=\left[\left(\frac{C_N}{s_n^\beta}\right)^{\frac{\beta-1}{\beta}}\!\! \!\!+\cO\left(\sin^{-\frac 1\beta}\gamma_ns_n^{\beta-1}\right)\right]\, \left[4\bar c s_n^{\beta-1}+\cO\(s_n^{2\beta-2}/\tan v_n\)\right]+\cO\left(\gamma_n^{-\frac 1\beta}\, s_n^{2\beta-2}\right)\nonumber\\
&=4\bar c (C_N)^{\frac{1}{\alpha}}+\cO(\gamma_n^{-\frac 1\beta}\, s_n^{2\beta-2})\nonumber\\
&=4\bar c (C_N)^{\frac{1}{\alpha}}+\cO(N^{-1} n^{-1})\, ,\label{vn+1B}
\end{align}
due to Proposition \ref{prop3}.
Recalling \eqref{hnB}, if we use a dummy variable and sum \eqref{vn+1B} from $1$ to $n$, we get
\beq\label{vnnB}w_n=4\bar c\, n\, C_N^{\frac{1}{\alpha}}+\cO(\ln n/N)\, .\eeq
In particular, for $n = N_2 = N/2 + \cO(1)$ we get 
\begin{eqnarray*}
\int_{0}^{\pi/2}(\sin u)^{1-\frac{1}{\beta}}\,du&=&\int_0^{\gamma_{N_2}}(\sin u)^{1-\frac{1}{\beta}}\, du+\cO\left(\gamma_n-\frac \pi 2\right)\\
&=&w_{N_2}+\cO(s_{N_2}^{\beta-1})=w_{N_2}+\cO(N^{-1}) =2\bar c NC_N^{\frac{1}{\alpha}}+\cO(\ln N/N).
\end{eqnarray*}
Thus
\begin{align}\label{eq:dNBB}
{C_N}{N^\alpha}=(2\bar c)^{-\alpha}\(\int_{0}^{\pi/2}(\sin u)^{1-\frac{1}{\beta}}\,du\)^{\alpha} +\cO( N^{-\alpha-2}\ln N)=\frac{I_1^{\alpha}}{(2\bar c)^{\alpha}}+\cO(\ln N/N)\, ,
\end{align}
and so
$$w_n=\frac{2I_1\,  n}N+\cO\left(\ln n/N\right)\, . $$
Similarly, one can show that
\begin{align}\label{eq:dNB'B}
{C'_N}{N^\alpha}=(2\bar c)^{-\alpha}I_1^{\alpha}+\cO(\ln N/N).
\end{align}
Let $N'$ be the number of reflections in the cusp (both on $\Gamma$ and on $\Gamma'$). Note that either $N'=2N$ or $N'=2N+1$, so that
\begin{align}\label{toto}
{C_N}{N'^\alpha}=\bar c^{-\alpha}I_1^{\alpha}+\cO(\ln N'/N').
\quad
\mbox{and}\quad
{C'_N}{N'^\alpha}=\bar c^{-\alpha}I_1^{\alpha}+\cO(\ln N'/N').
\end{align}
Let $\tilde M_m$ be the set of points in $\tilde M$ 
whose forward trajectories explore the cusp at $P$ during $m$ reflections off $\Gamma\cup \Gamma'$, before leaving the cusp.
Since $\nu$ is $T$-invariant,
\beq
\nu\left(\tilde M_m\right)=\frac 1 m\nu\left(\bigcup_{n=1}^mT^{n}\tilde M_m\right)\, .
\eeq
Let $y$ be an integer.
Observe that
$$\sum_{m'\ge y} 1_{\{N'\ge m'\}}=\sum_{m'\ge y}\sum_{m\ge m'} 1_{\{N'=m\}}=\sum_{m\ge y}\sum_{m'=y}^m 1_{\{N'=m\}}=\sum_{m\ge  y}(m-y+1) 1_{\{N'=m\}},$$ and so
\begin{eqnarray}
\sum_{m'\ge y}\nu\left(x\in \tilde M\ :\ N'\ge m'\right)
&=&\sum_{m\ge y}(m-y+1)\nu\left(\tilde M_m\right)\nonumber\\
&=&
-y\nu(x\in\tilde M\, :\, N'\ge y)+\nu\left(\bigcup_{m\ge y}\bigcup_{n=0}^{m}T^{n}\tilde M_m\right)\, .\label{KeyEQ}
\end{eqnarray}
But \eqref{hnB} and \eqref{eq:dNBB} together imply
that the set $\bigcup_{m\ge y}\bigcup_{n=1}^mT^{m}\tilde M_m$ corresponds to the set of points $(r,\varphi)$ based in the cusp
area $B_\epsilon(P)$ such that 
$$H(r,\varphi)=|r|^\beta\sin\varphi\le \bar c^{-\alpha}I_1^{\alpha}y^{-\alpha}+\cO(r^{2\beta-1})+\cO(y^{-1-\alpha}\ln y)\, .$$
Therefore the set $\bigcup_{m\ge y}\bigcup_{n=1}^mT^{m}\tilde M_m$ corresponds to the set of points $(r,\varphi)$ based in the cusp
area $B_\epsilon(P)$ such that 
$$r\le  \frac{\bar c^{-\frac\alpha\beta}I_1^{\frac\alpha\beta} y^{-\frac\alpha\beta}}{ \sin^{\frac{1}{\beta}}\varphi}+\cO\left(\frac{y^{-\alpha(\frac 1\beta-1)}}{\sin^{\frac 1\beta}\varphi}(\ln y/y+r^{2\beta-1})\right) .$$
Therefore
 \begin{align}
\nu\left(\bigcup_{m\ge y}\bigcup_{n=0}^mT^{m}\tilde M_m\right)\nonumber
&= 2\int_0^{\pi} 
\frac{\bar c^{-\frac\alpha\beta}I_1^{\frac\alpha\beta} y^{-\frac\alpha\beta}}{  \sin^{\frac{1}{\beta}-1}\varphi}\, d\varphi+\cO\left(y^{-1}\right)\nonumber\\
 &=  4\bar c^{1-\alpha}I_1^{\alpha} y^{1-\alpha}+\cO\left(y^{-1}\right)\, .\label{muRnB}
\end{align}
This tells us that $$\sum_{m\ge y}(m+1)\nu(\tilde M_m)\sim 4\bar c^{1-\alpha}I_1^{\alpha} y^{-\frac1{\beta-1}}$$ as $y$ goes to infinity.
Unfortunately we cannot apply directly the classical Tauberian theorem because 
we do not know if $m\nu(\tilde M_m)$ is decreasing or not.
But we do not want to estimate $\nu(\tilde M_m)$, we just want to estimate $$\mathfrak p_y:=\nu\left(x\in\tilde M\, :\, N'\ge y\right)=\sum_{m\ge y}\nu(\tilde M_m).$$
To this end, we adapt the Tauberian theorem argument as follows.
Due to \eqref{KeyEQ} and \eqref{muRnB} (considering the difference between the term of order $\lceil y\rceil$ and the term of order $\lceil(1+\varepsilon)y\rceil$), we know that, for every $\varepsilon>0$,
\beq\label{diffpy}
\lceil y\rceil \mathfrak p_{\lceil y\rceil}-\lceil (1+\varepsilon) y\rceil \mathfrak p_{\lceil (1+\varepsilon) y\rceil}+\sum_{m=\lceil y\rceil}^{\lceil (1+\varepsilon) y\rceil-1}\mathfrak p_m \sim
    4\bar c^{1-\alpha}I_1^{\alpha}\left(1-(1+\varepsilon)^{1-\alpha}\right) y^{1-\alpha}\, ,
\eeq
as $y$ goes to infinity.
Now using the fact that $(\mathfrak p_m)_m$ is decreasing, we obtain that
\beq\label{Tauberianestimate}
\lceil y\rceil \left(\mathfrak p_{\lceil y\rceil}-\mathfrak p_{\lceil (1+\varepsilon) y\rceil}\right)\le   \lceil y\rceil \mathfrak p_{\lceil y\rceil}-\lceil (1+\varepsilon) y\rceil \mathfrak p_{\lceil (1+\varepsilon) y\rceil}+\sum_{m=\lceil y\rceil}^{\lceil (1+\varepsilon) y\rceil}\mathfrak p_m\le 
\lceil (1+\varepsilon) y\rceil \left(\mathfrak p_{\lceil y\rceil}-\mathfrak p_{\lceil (1+\varepsilon) y\rceil}\right)\, .
\eeq
Fix an arbitrary $\vartheta\in(0,1)$.
Choose $\varepsilon\in(0,\vartheta)$ small enough so that 
$$ (\alpha-1)\varepsilon(1-\vartheta)\le 1-(1+\varepsilon)^{1-\alpha}\le  (\alpha-1)\varepsilon(1+\vartheta)$$
and
$$ \alpha\varepsilon(1-\vartheta)\le 1-(1+\varepsilon)^{-\alpha}\le  \alpha\varepsilon(1+\vartheta)\, .$$
Due to \eqref{diffpy}, for every $y$ large enough, we know that
$$(\alpha-1)\varepsilon(1-\vartheta)^2 4\bar c^{1-\alpha}I_1^{\alpha} y^{1-\alpha}\le \lceil y\rceil \mathfrak p_{\lceil y\rceil}-\lceil (1+\varepsilon) y\rceil \mathfrak p_{\lceil (1+\varepsilon) y\rceil}+\sum_{m=y}^{\lceil (1+\varepsilon) y\rceil}\mathfrak p_m \le (\alpha-1)\varepsilon(1+\vartheta)^2
    4\bar c^{1-\alpha}I_1^{\alpha} y^{1-\alpha}\, ,$$
and so, due to \eqref{Tauberianestimate}, we obtain that
$$(\alpha-1)\varepsilon\frac{(1-\vartheta)^2}{1+\vartheta} 4\bar c^{1-\alpha}I_1^{\alpha} y^{-\alpha}\le \mathfrak p_{\lceil y\rceil}-\mathfrak p_{\lceil (1+\varepsilon) y\rceil} \le (\alpha-1)\varepsilon(1+\vartheta)^2
    4\bar c^{1-\alpha}I_1^{\alpha} y^{-\alpha}\, ,$$
for every $y$ large enough.
Now using the fact that $\mathfrak p_y=\mathfrak p_{\lceil y\rceil}=\sum_{k\ge 0}\left(p_{\lceil (1+\varepsilon)^k y\rceil}-p_{\lceil (1+\varepsilon)^{k+1}y\rceil}\right)$, it follows that, for every $y$ large enough,
$$(\alpha-1)\varepsilon\frac{(1-\vartheta)^2}{1+\vartheta} 4\bar c^{1-\alpha}I_1^{\alpha} \sum_{k\ge 0}(1+\varepsilon)^{-k\alpha} y^{-\alpha}\le \mathfrak p_y \le (\alpha-1)\varepsilon(1+\vartheta)^2
    4\bar c^{1-\alpha}I_1^{\alpha} \sum_{k\ge 0}(1+\varepsilon)^{-k\alpha}y^{-\alpha}\, $$
which can be rewritten
$$(\alpha-1)\frac\varepsilon{1-(1+\varepsilon)^{-\alpha}}\frac{(1-\vartheta)^2}{1+\vartheta} 4\bar c^{1-\alpha}I_1^{\alpha} y^{-\alpha}\le \mathfrak p_y \le (\alpha-1)\frac\varepsilon{1-(1+\varepsilon)^{-\alpha}}(1+\vartheta)^2
    4\bar c^{1-\alpha}I_1^{\alpha} y^{-\alpha}\, .$$
Using our second condition on $\varepsilon$, this leads to
$$\frac{\alpha-1}\alpha\frac{(1-\vartheta)^2}{(1+\vartheta)^2} 4
\bar c^{1-\alpha}I_1^{\alpha} y^{-\alpha}\le \mathfrak p_y \le \frac{\alpha-1}\alpha\frac{(1+\vartheta)^2}{1-\vartheta}
    4\bar c^{1-\alpha}I_1^{\alpha} y^{-\alpha}\, ,$$
for every $y$ large enough.
Since $\vartheta$ is arbitrary, we have proved that
$$\mathfrak p_y\sim  4\frac{\alpha-1}\alpha\bar c^{1-\alpha}I_1^{\alpha} y^{-\alpha}=4\beta^{-1}\bar c^{1-\alpha}I_1^{\alpha} y^{-\alpha}\, ,$$
as $y$ goes to infinity,
which completes the proof of Lemma \ref{Mm2}.
\begin{proof}[Proof of Lemma \ref{lem2}]
We also now obtain the proof of Lemma \ref{lem2} by using  \eqref{vnnB} and \eqref{toto} in place of (6.8) and (6.9) in the proof of Lemma 4.4 in \cite{JZ17}, and then making the appropriate obvious adjustments.
\end{proof}

\section{Skorokhod $J_1$ and $M_1$ topologies}
A stronger result than the limit theorem in \eqref{stablelaw} is its functional version, called a functional limit theorem or weak invariance principle.
The $W_n$'s are elements in the Skorokhod space $\D[0,\infty)$, i.e., the space of all functions $\phi$ on $[0,\infty)$ that are right-continuous and have left-hand limits $\phi(t-)$ for every $t>0$.

We will consider two different topologies on $\D[0,\infty)$. The most commonly used topology in the literature is
Skorokhod's $J_1$-topology which is described as follows: if $\phi_n, \phi\in\D[0,\infty)$, then $\phi_n\to\phi$ in the $J_1$-topology if and only if there exists a sequence of $\{\lambda_n\}\subset\cA$, such that $$\sup_s|\lambda_n(s)-s|\to 0,\,\,\,\,\sup_{s\leq m}|\phi_n(\lambda_n(s))-\phi(s)|\to 0$$for all $m\in\mathbb{N}$, where $\cA$ is the family of strictly increasing, continuous mappings $\lambda$ of $[0,\infty]$ onto itself such that $\lambda(0)=0$ and $\lambda(\infty)=\infty$.

In contrast, the $M_1$-topology allows a function $\phi_1$ with a jump at $t$ to be approximated arbitrarily well by some continuous $\phi_2$ (with large slope near $t$).
The metric $d_{M_1}$
that generates the $M_1$-topology
on $\D[0, \infty)$ is defined using completed graphs. For $\phi\in \D[0, \infty]$ the completed
graph of $\phi$ is the set
 $$\Gamma(\phi):= \{(t, z) \in [0, \infty) \times \mathbb{R} : z = \lambda\phi(t-) + (1- \lambda)\phi(t) \text{ for some } \lambda\in [0, 1]\}$$
where $\phi(t-)$ is the left limit of $\phi$ at $t$. Besides the points of the graph
$\{(t, \phi(t)) :t \in [0, \infty)\}$, the completed graph of $\phi$ also contains the vertical line
segments joining $(t, \phi(t))$ and $(t, \phi(t-))$ for all discontinuity points $t$ of $\phi$.
We define an order on the graph $\Gamma(\phi)$ by saying that $(t_1, z_1) \leq (t_2, z_2)$ if either
(i) $t_1 < t_2$ or (ii) $t_1 = t_2$ and $|\phi(t_1-) -z_1| \leq  |\phi(t_2-) - z_2|$. A parametric
representation of the completed graph $\Gamma(\phi)$ is a continuous nondecreasing
function $(s, y)$ mapping $[0, \infty)$ onto $\Gamma(\phi)$, with $s$ being the time component
and $y$ being the spatial component. Let $\Lambda(\phi)$ denote the set of parametric
representations of the graph $\Gamma(\phi)$. For $\phi_1, \phi_2\in \D[0, \infty)$ define
$$d_{M_1}
(\phi_1, \phi_2) := \inf\{\|s_1 - s_2\|_{[0,\infty)} \vee\|u_1 - u_2\|_{[0,\infty)} : (s_i
, u_i) \in \Lambda(\phi_i), i = 1, 2\},$$
where $\|\phi\|_{[0,\infty)}= \sup\{|\phi(t)| :t \in [0, \infty)\}$. This definition introduces $d_{M_1}$
as a metric
on $\D[0, \infty)$. The induced topology is called Skorokhod's $M_1$-topology and
is weaker than the more frequently used $J_1$-topology.

\end{appendix}
\section*{Acknowledgements}

The research of H. Zhang was supported in
part by NSF CAREER grant DMS-1151762. The research of P. Jung was supported in part by NRF grant N01170220 from the Republic of Korea.

\bibliographystyle{alpha}
\bibliography{BibJuly15}

\begin{thebibliography}{BCD11}

\bibitem[AT92]{avram1992weak}
F.~Avram and M.S. Taqqu.
\newblock Weak convergence of sums of moving averages in the $\alpha$-stable
  domain of attraction.
\newblock {\em The Annals of Probability}, pages 483--503, 1992.

\bibitem[BCD11]{balint2011limit}
P{\'e}ter B{\'a}lint, Nikolai Chernov, and Dmitry Dolgopyat.
\newblock Limit theorems for dispersing billiards with cusps.
\newblock {\em Communications in mathematical physics}, 308(2):479, 2011.

\bibitem[Ber27]{bernstein1927extension}
Serge Bernstein.
\newblock Sur l'extension du th{\'e}or{\`e}me limite du calcul des
  probabilit{\'e}s aux sommes de quantit{\'e}s d{\'e}pendantes.
\newblock {\em Mathematische Annalen}, 97(1):1--59, 1927.

\bibitem[Bil68]{billingsley1968convergence}
P.~Billingsley.
\newblock {\em {Convergence of probability measures}}.
\newblock Wiley New York, 1968.

\bibitem[Bil99]{billingsley1999convergence}
P.~Billingsley.
\newblock {\em Convergence of Probability Measures}.
\newblock Wiley, New York, 1999.

\bibitem[CM06]{chernov2006chaotic}
Nikolai Chernov and Roberto Markarian.
\newblock {\em Chaotic billiards}.
\newblock American Mathematical Society, 2006.

\bibitem[CM07]{chernov2007dispersing}
N~Chernov and R~Markarian.
\newblock Dispersing billiards with cusps: slow decay of correlations.
\newblock {\em Communications in mathematical physics}, 270(3):727--758, 2007.

\bibitem[CZ05]{chernov2005billiards}
Nikolai Chernov and Hong-Kun Zhang.
\newblock Billiards with polynomial mixing rates.
\newblock {\em Nonlinearity}, 18(4):1527, 2005.

\bibitem[CZ09]{CZ09}
Nikolai Chernov and Hong-Kun Zhang.
\newblock On statistical properties of hyperbolic systems with singularities.
\newblock {\em Journal of Statistical Physics}, 136(4):615--642, 2009.

\bibitem[DR78]{durrett1978functional}
Richard Durrett and Sidney~I Resnick.
\newblock Functional limit theorems for dependent variables.
\newblock {\em The Annals of Probability}, pages 829--846, 1978.

\bibitem[JZ18]{JZ17}
Paul Jung and Hong-Kun Zhang.
\newblock Stable laws for chaotic billiards with cusps at flat points.
\newblock {\em Annales Henri Poincar\'e}, 2018.

\bibitem[Kal73]{kallenberg1973characterization}
Olav Kallenberg.
\newblock Characterization and convergence of random measures and point
  processes.
\newblock {\em Probability Theory and Related Fields}, 27(1):9--21, 1973.

\bibitem[Mac83]{machta1983power}
Jonathan Machta.
\newblock Power law decay of correlations in a billiard problem.
\newblock {\em Journal of statistical physics}, 32(3):555--564, 1983.

\bibitem[Mar04]{markarian2004billiards}
Roberto Markarian.
\newblock Billiards with polynomial decay of correlations.
\newblock {\em Ergodic Theory and Dynamical Systems}, 24(01):177--197, 2004.

\bibitem[MV18]{melbourne2018convergence}
Ian Melbourne and Paulo Varandas.
\newblock Convergence to a l$\backslash$'evy process in the skorohod $ m\_1 $
  and $ m\_2 $ topologies for nonuniformly hyperbolic systems, including
  billiards with cusps.
\newblock {\em arXiv preprint arXiv:1809.06572}, 2018.

\bibitem[MZ15]{melbourne2015weak}
Ian Melbourne and Roland Zweim{\"u}ller.
\newblock {Weak convergence to stable L{\'e}vy processes for nonuniformly
  hyperbolic dynamical systems}.
\newblock {\em Annales de l'Institut Henri Poincar{\'e}, Probabilit{\'e}s et
  Statistiques}, 51(2):545--556, 2015.

\bibitem[PS18]{FPBSPoisson2}
Fran{\c{c}}oise P{\`e}ne and Benoit Saussol.
\newblock {Spatio-temporal Poisson processes for visits to small sets}.
\newblock {\em arXiv preprint arXiv:1803.06865}, 2018.

\bibitem[Res87]{resnick1987extreme}
Sidney~I. Resnick.
\newblock {\em Extreme values, regular variation, and point processes}.
\newblock Springer, 1987.

\bibitem[Ser70]{serfling1970moment}
Robert~J Serfling.
\newblock Moment inequalities for the maximum cumulative sum.
\newblock {\em The Annals of Mathematical Statistics}, pages 1227--1234, 1970.

\bibitem[ST94]{samorodnitsky1994stable}
G.~Samorodnitsky and M.S. Taqqu.
\newblock {\em {Stable non-Gaussian random processes: stochastic models with
  infinite variance}}.
\newblock Chapman \& Hall/CRC, 1994.

\bibitem[TK10]{tyran2010weak}
Marta Tyran-Kami{\'n}ska.
\newblock Weak convergence to l{\'e}vy stable processes in dynamical systems.
\newblock {\em Stochastics and Dynamics}, 10(02):263--289, 2010.

\bibitem[Zha17]{Z2016b}
Hong-Kun Zhang.
\newblock {Decay of correlations for billiards with flat points II: cusps
  effect}.
\newblock {\em {Contemporary Mathematics}}, 2017.

\end{thebibliography}

\end{document}